\documentclass[10pt, a4paper]{article}
\usepackage{amsfonts}
\usepackage{mathrsfs}
\usepackage{latexsym}
\usepackage{xy}
\usepackage{amsfonts,amsmath,amssymb,amsthm}
\usepackage{color}

\xyoption{all}

\newcommand{\bcen}{\begin{center}}     \newcommand{\ecen}{\end{center}}
\newcommand{\bay}{\begin{array}}      \newcommand{\eay}{\end{array}}
\newcommand{\beq}{\begin{eqnarray*}}      \newcommand{\eeq}{\end{eqnarray*}}

\def\Coder{\mathrm{Coder}}

\def\comp{\mathrm{comp}}

\def\Der{\mathrm{Der}}

\def\dz{\delta}
\def\End{\mathrm{End}}

\def\ev{\mathrm{ev}}

\def\Hom{\mathrm{Hom}}

\def\id{\mathrm{id}}
\def\Im{\mathrm{Im}}

\def\Ker{\mathrm{Ker}}

\def\op{\mathrm{op}}

\def\per{\mathrm{per}}
\def\proj{\mathrm{proj}}

\def\rad{\mathrm{rad}}

\def\RHom{\mathrm{RHom}}

\def\Tw{\mathrm{Tw}}

\begin{document}

\newtheorem{theorem}{Theorem}
\newtheorem{proposition}{Proposition}
\newtheorem{lemma}{Lemma}
\newtheorem{corollary}{Corollary}
\newtheorem{remark}{Remark}
\newtheorem{example}{Example}
\newtheorem{definition}{Definition}
\newtheorem*{conjecture}{Conjecture}
\newtheorem{question}{Question}

\title{Hochschild (co)homologies of dg $K$-rings and their Koszul duals}

\author{Yang Han, Xin Liu and Kai Wang}

\date{\footnotesize KLMM, Academy of Mathematics and Systems Science,
Chinese Academy of Sciences, \\ Beijing 100190, China. \\ School of Mathematical Sciences, University of
Chinese Academy of Sciences, \\ Beijing 100049, China.\\ E-mail:
hany@iss.ac.cn (Y. Han), liuxin215@mails.ucas.ac.cn (X. Liu),\\ wangkai@amss.ac.cn (K. Wang)}

\maketitle

\begin{abstract} We formulate the (co)bar construction theory of dg $K$-(co)rings
and the calculus theory of the Hochschild homology and cohomology of dg $K$-rings.
As applications, we compare the Hochschild (co)homologies of a complete typical dg $K$-ring and its Koszul dual.
Moreover, we show that the Koszul dual of a finite dimensional complete typical $d$-symmetric dg $K$-ring is a $d$-Calabi-Yau dg algebra
whose Hochschild cohomology is a Batalin-Vilkovisky algebra.
Furthermore, we prove that the Hochschild cohomologies of a finite dimensional complete typical $d$-symmetric dg $K$-ring
and its Koszul dual are isomorphic as Batalin-Vilkovisky algebras.
In conclusion, we found a connection between the Batalin-Vilkovisky algebra structures
on the Hochschild cohomologies of $d$-Calabi-Yau dg algebras and $d$-symmetric dg $K$-rings.
\end{abstract}

\medskip


\medskip


\bigskip

\tableofcontents

\section{Introduction}

Throughout this paper, $k$ is a field, $K=k^t$ for a positive integer $t$, and all algebras are over $k$.
Gerstenhaber showed that the Hochschild cohomology $HH^\bullet(A)$ of an associative algebra $A$ is a {\it Gerstenhaber algebra},
i.e., $HH^\bullet(A)$ is a graded commutative associative algebra on the cup product $\cup$,
$HH^{\bullet+1}(A)$ is a graded Lie algebra on the Gerstenhaber bracket $[-,-]$,
and $[-,-]$ is a graded derivation on the cup product in each variable \cite{Ger63}.
For some kinds of associative algebras $A$, $HH^\bullet(A)$ is even a {\it Batalin-Vilkovisky (BV for short) algebra},
i.e., there is an operator $\Delta$ on $HH^\bullet(A)$ of degree $-1$ such that
$\Delta^2=0$ and the Gerstenhaber bracket $[-,-]$ is just the deviation of $\Delta$ from being a derivation (See \cite{Hue98,Rog09} and the references therein).
In this aspect, on one hand, Ginzburg proved that the Hochschild cohomology of a Calabi-Yau (CY for short) algebra is a BV algebra \cite{Gin06}.
Abbaspour generalized Ginzburg's result to some $d$-CY differential graded (dg for short) algebras \cite{Abb15},
and Kowalzig and Kr\"{a}hmer generalized it to some twisted CY algebras \cite{KowKra14}.
On the other hand, Tradler showed that the Hochschild cohomology of a symmetric algebra is a BV algebra \cite{Tra08,Men09,Men04}.
Abbaspour generalized Tradler's result to {\it derived Poincar\'{e} duality algebras} which can be viewed as some nice {\it derived $d$-symmetric dg algebras}.
Moreover, the BV algebra structures on the Hochschild cohomologies of some Frobenius algebras were studied in \cite{Yan13,IvaIvaVolZho15,LiuZho16,LamZhoZim16,Vol16}.

Whereas there is a correspondence between some kinds of exact $d$-CY dg algebras and {\it cyclic $A_\infty$-algebras of degree $d$},
which can be viewed as {\it $d$-symmetric $A_\infty$-algebras}, under Koszul duality \cite{VdB15},
there should be some relations between the BV algebra structures on the Hochschild cohomologies
of some kinds of $d$-CY dg algebras and $d$-symmetric dg algebras.
In fact, Chen, Yang and Zhou proved that a Koszul (graded) algebra is $d$-CY if and only if its Koszul dual is a cyclic algebra of degree $d$,
and the Hochschild cohomologies of a Koszul $d$-CY algebra and its Koszul dual are isomorphic as BV algebras \cite{CheYanZho16}.
In this paper, we study such a relation in a more general setting.

Recall that a dg algebra $A$ is {\it locally finite} if every homogeneous component $A_i$ of $A$ is finite dimensional.
A locally finite augmented dg algebra $A$ is {\it typical} if it is either non-negative or non-positive simply connected (i.e., $A_0=k$ and $A_{-1}=0$).
Most of important dg algebras from algebra and topology are typical.
Usually, one has two ways to define the Koszul dual of a locally finite augmented dg algebra.
One is the cobar construction $\Omega A^\vee$ of the graded $k$-dual $A^\vee$ of $A$ as defined in \cite{VdB15,Kel03},
the other is the graded $k$-dual $(BA)^\vee$ of the bar construction $BA$ of $A$ as defined in \cite{Kel94}.
The typicality of a dg algebra above ensures these two definitions of Koszul dual coincide up to isomorphism.
Given a {\it $d$-symmetric dg algebra} $A$, i.e., an augmented dg algebra admitting a dg $A$-bimodule quasi-isomorphism from $A$ to $A^\vee[-d]$.
We wish we could show the Koszul dual $\Omega A^\vee$ of $A$
is a $d$-CY dg algebra and the Hochschild cohomologies of $\Omega A^\vee$ and $A$ are isomorphic as BV algebras.
To this end, we have to prove $\Omega A^\vee$ is homologically smooth \cite{KonSoi09}.
It is here that we need the condition that $A$ is complete.
Recall that a finite dimensional algebra $A$ is {\it elementary}
if the factor algebra $A/\rad A$ of $A$ modulo its Jacobson radical $\rad A$ is isomorphic to $k^t$ for a positive integer $t$ (See \cite{AusReiSma95}).
In fact, if $k$ is an algebraically closed field,
every finite dimensional $k$-algebra is Morita equivalent to a finite dimensional elementary algebra \cite{AssSimSko06}.
A finite dimensional elementary algebra must be isomorphic to a bound quiver algebra \cite{AusReiSma95},
thus one can study its representation theory by the combinatorics, geometry and topology of the quivers.
Note that (bound) quiver algebras play quite important roles in representation theory of algebras \cite{AssSimSko06} and
Calabi-Yau algebra theory \cite{Gin06}. A bound quiver algebra $A=kQ/I$, where $Q$ is a finite quiver and
$I$ is an admissible ideal of the path algebra $kQ$, is certainly an augmented algebra
whose augmentation could be given by a natural projection $A \twoheadrightarrow A/\rad A \twoheadrightarrow k$.
However, $A$ is not complete unless $Q$ is a quiver with only one vertex,
which causes some important known results can not be applied to bound quiver algebras directly,
especially, the difficulty in proving the homological smoothness of $\Omega A^\vee$.
Nonetheless, a bound quiver algebra $kQ/I$ is always a complete $K$-ring where $K=kQ_0 \cong k^{|Q_0|}$.
Now that we hope our main results are applicable to all ordinary symmetric algebras, at least all elementary symmetric algebras,
we have to study (dg) $K$-(co)rings introduced by Sweedler \cite{Swe75} instead of (dg) $k$-algebras.

In this paper, firstly, we will formulate the (co)bar construction theory of dg $K$-(co)rings.
Note that such bar constructions for bound quiver algebras have been already studied by Cibils in \cite{Cil90},
and such (co)bar constructions for augmented dg $K$-(co)rings have been investigated by He in his doctoral thesis \cite{He04}.
Here, we will generalize, complete and improve their results.
Secondly, we will formulate the calculus theory of the Hochschild homology and cohomology of augmented dg $K$-rings mimicking the augmented dg $k$-algebra case in \cite{TamTsy05,TamTsy00,CunSkaTsy04,GelDalTsy89}.
Now that we will adopt the {\it $K$-reduced} two-sided bar resolution in the sense of Cibils \cite{Cil90}
instead of the ordinary two-sided bar resolution of an augmented $k$-algebra, this formulation is necessary.
Furthermore, we will apply the theories above to compare the Hochschild (co)homologies
of a complete typical dg $K$-ring $A$ and its Koszul dual $\Omega A^\vee$.
We will show that $HH^\bullet(\Omega A^\vee) \cong HH^\bullet(A)$, $HH_{-\bullet}(\Omega A^\vee) \cong HH_\bullet(A)^\vee$,
and the latter is compatible with Connes operators (Theorem \ref{Theorem-HHHH_Connes}), which generalize some results in \cite{FelMenTho05}.
Moreover, we will prove that the Koszul dual $\Omega A^\vee$ of a finite dimensional complete typical $d$-symmetric dg $K$-ring $A$
is a $d$-CY dg algebra (Theorem \ref{Theorem-Sym-Dual-CY}).
For later use, we will revise the results that the Hochschild cohomologies of some $d$-CY dg algebras and $d$-symmetric dg $K$-rings
are BV algebras, such that all involved operators are defined in $K$-reduced version.
Finally, we prove that the Hochschild cohomology $HH^\bullet(\Omega A^\vee)$ of the Koszul dual $\Omega A^\vee$
of a finite dimensional complete typical $d$-symmetric dg $K$-ring $A$ is a BV algebra,
and $HH^\bullet(\Omega A^\vee)$ and $HH^\bullet(A)$ are isomorphic as BV algebras (Theorem \ref{Theorem-SymCY-BVIso}).
Restricted to ordinary algebras, our results imply if $A$ is an elementary symmetric algebra then
the Koszul dual $\Omega A^\vee$ is a 0-CY dg algebra, $HH^{\bullet}(\Omega A^\vee)$ is a BV algebra,
and $HH^\bullet(A)$ and $HH^{\bullet}(\Omega A^\vee)$ are isomorphic as BV algebras.
In conclusion, we found a connection between BV algebra structures
on the Hochschild cohomologies of $d$-CY dg algebras and $d$-symmetric dg $K$-rings.

We refer to \cite{FelHalTho01,AvrFoxHal00} for dg algebras and dg homological algebra.
For derived categories of dg algebras, we refer to \cite{Kel94}.
In the differential graded situation, we always use the {\it Koszul sign rule},
i.e., we add the sign $(-1)^{|x||y|}$ once we exchange the positions of two graded objects $x$ and $y$ of degree $|x|$ and $|y|$ respectively.
For example, if $f : V \rightarrow V'$ and $g : W \rightarrow W'$ are two morphisms of graded $k$-vector spaces
then the tensor product $f \otimes g : V \otimes_k W \rightarrow V' \otimes_k W'$ is given by
$(f \otimes g)(v \otimes w) = (-1)^{|g||v|} f(v) \otimes g(w)$ for all $v \in V$ and $w \in W$.
By convention, in a dg $k$-module, an element of lower degree $i \in \mathbb{Z}$ is of upper degree $-i$.
Unless stated otherwise, $\otimes := \otimes_K$.

\section{(Co)bar constructions of dg $K$-(co)rings}

In this section, we formulate the (co)bar construction theory of dg $K$-(co)rings.
In the case of $t=1$, i.e., $K=k$, it is just the classical (co)bar construction theory of $k$-(co)algebras (See \cite{LodVal12,HusMooSta74,FelHalTho01}).

\subsection{Dg $K$-(co)rings}

The concepts of $K$-ring and $K$-coring were introduced by Sweedler in \cite{Swe75}.

\bigskip

\noindent{\bf Dg $K$-rings.} A (unital, associative) {\it $K$-ring} $A = (A, \mu, \eta)$ is a $K$-bimodule $A$ equipped with
two $K$-bimodule morphisms $\mu: A \otimes A \rightarrow A$, called
{\it product}, and $\eta : K \rightarrow A$, called {\it unit},
satisfying $\mu \circ (\mu \otimes \id) = \mu \circ (\id \otimes \mu)$ and $\mu \circ (\eta \otimes \id) = \id = \mu \circ (\id \otimes \eta)$,
i.e., the following two diagrams are commutative:

\medskip

associativity \hspace{20mm} $\xymatrix{ A \otimes A \otimes A \ar[r]^-{\id \otimes \mu} \ar[d]_-{\mu \otimes \id} & A \otimes A \ar[d]^-{\mu} \\
A \otimes A \ar[r]^-{\mu} & A }$

\medskip

unitality \hspace{22mm} $\xymatrix{ & A \otimes A \ar[d]^-{\mu} & \\ K \otimes A \ar[ur]^-{\eta \otimes \id} \ar@{=}[r] & A & A \otimes K. \ar[ul]_-{\id \otimes \eta} \ar@{=}[l] }$

\medskip

\noindent A {\it morphism} from a $K$-ring $(A , \mu, \eta)$ to a $K$-ring $(A' , \mu', \eta')$ is a $K$-bimodule morphism
$f : A \rightarrow A'$ satisfying $f \circ \mu = \mu' \circ (f \otimes f)$ and $f \circ \eta = \eta'$, i.e., the following two diagrams are commutative:
$$\xymatrix{ A \otimes A \ar[r]^-{\mu} \ar[d]_-{f \otimes f} & A \ar[d]^-{f} \\
A' \otimes A' \ar[r]^-{\mu'} & A' } \qquad\qquad \xymatrix{ K \ar[r]^-{\eta} \ar[dr]_-{\eta'} & A \ar[d]^-{f} \\ & A'. }$$
A {\it graded $K$-ring} $A = (A, \mu, \eta)$ is both a graded $K$-bimodule
$A = \bigoplus\limits_{i \in \mathbb{Z}} A_i$ and a $K$-ring $(A, \mu, \eta)$ such that both $\mu$ and $\eta$ are graded $K$-bimodule morphisms.
A {\it morphism of graded $K$-rings} is both a morphism of graded $K$-bimodules and a morphism of $K$-rings.
A {\it differential graded (dg) $K$-ring} $A = (A, \mu, \eta, d)$ is both a dg $K$-bimodule $(A,d)$
whose differential $d : A \rightarrow A$ is of degree $-1$,
and a graded $K$-ring $(A, \mu, \eta)$, such that both $\mu$ and $\eta$ are morphisms of dg $K$-bimodules.
A {\it morphism of dg $K$-rings} is both a morphism of dg $K$-bimodules and a morphism of graded $K$-rings.
An {\it augmented dg $K$-ring} $A = (A, \mu, \eta, d, \varepsilon)$ is a dg $K$-ring $(A, \mu, \eta, d)$ equipped
with a dg $K$-ring morphism $\varepsilon : A \rightarrow K$, called the {\it augmentation} of $A$, satisfying $\varepsilon \circ \eta = \id_K$.
If $A$ is an augmented dg $K$-ring then $A = K 1_A \oplus \overline{A}$ where $1_A=\eta(1_K)$ and $\overline{A} = \Ker \varepsilon$,
called the {\it augmentation ideal} of $A$.
We always identify $K 1_A$ with $K$.
A {\it morphism of augmented dg $K$-rings} is a morphism of dg $K$-rings $f : A \rightarrow A'$ satisfying $\varepsilon_{A'} \circ f = \varepsilon_A$.

\bigskip

\noindent{\bf Dg modules.} Let $(A, \mu, \eta)$ be a $K$-ring. A {\it left $A$-module} $M = (M,\mu^l)$ is a left $K$-module $M$ equipped
with a left $K$-module morphism $\mu^l : A \otimes M \rightarrow M$, called {\it left action},
satisfying $\mu^l \circ (\mu \otimes \id_M) = \mu^l \circ (\id_A \otimes \mu^l)$
and $\mu^l \circ (\eta \otimes \id_M) = \id_M$, i.e., the following two diagrams are commutative:

\medskip

associativity \hspace{20mm} $\xymatrix{ A \otimes A \otimes M \ar[r]^-{\id_A \otimes
\mu^l} \ar[d]_-{\mu \otimes \id_M} & A \otimes M \ar[d]^-{\mu^l} \\ A \otimes M \ar[r]^-{\mu^l} & A }$

\medskip

unitality \hspace{33mm} $\xymatrix{ & A \otimes M \ar[d]^-{\mu^l} \\ K \otimes M \ar[ur]^-{\eta \otimes \id_M} \ar@{=}[r] & M. }$

\medskip

\noindent Similarly, we can define a {\it right $A$-module} $M$ with a {\it right action} $\mu^r : M \otimes A \rightarrow M$.
Let $A$ and $B$ be two $K$-rings. An {\it $A$-$B$-bimodule} $M = (M, \mu^l, \mu^r)$
is not only a $K$-bimodule but also a left $A$-module $(M,\mu^l)$ and a right $B$-module $(M,\mu^r)$,
such that $\mu^r \circ (\mu^l \otimes \id_B) = \mu^l \circ (\id_A \otimes \mu^r)$,
i.e., the following diagram is commutative:
$$\xymatrix{ A \otimes M \otimes B \ar[r]^-{\id_A \otimes \mu^r} \ar[d]^-{\mu^l \otimes \id_B} & A \otimes M \ar[d]^-{\mu^l} \\
M \otimes B \ar[r]^-{\mu^r} & M. }$$
A {\it morphism} from a left $A$-module $(M, \mu^l_M)$ to a left $A$-module $(M',\mu^l_{M'})$ is a left $K$-module morphism
$f : M \rightarrow M'$ satisfying $f \circ \mu^l_M = \mu^l_{M'} \circ (\id_A \otimes f).$
Similarly, we can define a {\it morphism of right modules} and a {\it morphism of bimodules}.
Let $A$ be a graded $K$-ring. A {\it graded left $A$-module} $M = (M, \mu^l)$
is both a graded left $K$-module $M = \oplus_{i \in \mathbb{Z}}M_i$ and a left $A$-module $(M, \mu^l)$ such that $\mu^l$ is a morphism of graded left $K$-modules.
A {\it morphism of graded left $A$-modules} is both a morphism of graded left $K$-modules and a morphism of left $A$-modules.
Similarly, we can define a {\it graded right $A$-module} and a {\it morphism of graded right $A$-modules}.
Let $A$ and $B$ be two graded $K$-rings. A {\it graded $A$-$B$-bimodule} $M = (M, \mu^l, \mu^r)$ is
both a graded $K$-bimodule $M = \bigoplus\limits_{i \in \mathbb{Z}}M_i$ and an $A$-$B$-bimodule $(M, \mu^l, \mu^r)$
such that both $\mu^l$ and $\mu^r$ are morphisms of graded $K$-bimodules.
A {\it morphism of graded $A$-$B$-bimodules} is both a morphism of graded $K$-bimodules and a morphism of $A$-$B$-bimodules.
Let $A$ be a dg $K$-ring. A {\it dg left $A$-module} $M = (M, \mu^l, d_M)$ is
both a dg left $K$-module $(M,d_M)$, whose differential $d_M : M \rightarrow M$ is of degree $-1$, and a graded left $A$-module $(M, \mu^l)$,
such that $\mu^l$ is a morphism of dg left $K$-modules.
A {\it morphism of dg left $A$-modules} is both a morphism of dg left $K$-modules and a morphism of graded left $A$-modules.
Similarly, we can define a {\it dg right $A$-module} and a {\it morphism of dg right $A$-modules}.
Let $A$ and $B$ be two dg $K$-rings. A {\it dg $A$-$B$-bimodule} $M = (M, \mu^l, \mu^r, d)$ is both
a dg $K$-bimodule $(M,d)$ and a graded $A$-$B$-bimodule $(M, \mu^l, \mu^r)$ such that both $\mu^l$ and $\mu^r$ are morphisms of dg $K$-bimodules.
A {\it morphism of dg $A$-$B$-bimodules} is both a morphism of dg $K$-bimodules and a morphism of graded $A$-$B$-bimodules.
For dg $k$-modules $M$ and $M'$, the {\it switching map} $\tau_{M,M'} : M \otimes_k M' \rightarrow M' \otimes_k M$
is given by $\tau_{M,M'}(m \otimes m') := (-1)^{|m||m'|} m' \otimes m$ for all $m \in M$ and $m' \in M'$,
which is a dg $k$-module isomorphism.
For dg $K$-bimodules $M$ and $M'$, we can define the {\it switching map} $M \otimes_{K^e} M' \rightarrow M' \otimes_{K^e} M$ similarly.
It is still denoted by $\tau_{M,M'}$, which can not cause any confusion.

\begin{remark}{\rm A dg $K$-ring $A$ could be viewed as the dg $k$-category $\mathcal{A}$ with objects $e_1,\cdots,e_t$
which is the canonical $k$-basis of $K=k^t$, and Hom-sets $\Hom_\mathcal{A}(e_i,e_j):=e_jAe_i,\ \forall 1\leq i,j\leq t$.
Moreover, a dg right $A$-module is nothing but a dg functor from $\mathcal{A}^\op$ to $\mbox{\rm Dif}k$ (See \cite{Kel94}).
Therefore, we have the derived category theory of dg $K$-rings.
}\end{remark}

\bigskip

\noindent{\bf Opposite dg $K$-rings.}  Let $A$ be a dg $K$-ring. Then $A$ is a dg $K$-bimodule with the left action $\mu^l$ and the right action $\mu^r$.
Denote by $\tilde{A}$ the dg $K$-bimodule $A$ with the left action $\tilde{\mu}^l : K \otimes_k A \xrightarrow{\tau_{K,A}} A \otimes_k K \xrightarrow{\mu^r} A$
and the right action $\tilde{\mu}^r : A \otimes_k K \xrightarrow{\tau_{A,K}} K \otimes_k A \xrightarrow{\mu^l} A$.
Then $\tilde{A}$ is a dg $K$-ring with product $\tilde{\mu} : \tilde{A} \otimes_K \tilde{A} \xrightarrow{\tau} A \otimes_K A \xrightarrow{\mu} A$
where $\tau : \tilde{A} \otimes_K \tilde{A} \rightarrow A \otimes_K A, \ a \otimes b \mapsto (-1)^{|a||b|}b \otimes a$,
unit $\tilde{\eta} : K \rightarrow \tilde{A}, \ x \mapsto \eta(x)$,
and differential $\tilde{d} : \tilde{A} \rightarrow \tilde{A},\ a \mapsto d(a)$.
The dg $K$-ring $(\tilde{A},\tilde{\mu},\tilde{\eta})$ is called the {\it opposite dg $K$-ring} of $A$, denoted by $A^\op$.
If $A$ is an augmented dg $K$-ring then the opposite dg $K$-ring $A^\op$ of $A$
automatically admits an augmentation $\tilde{\varepsilon} : \tilde{A} \rightarrow K, \ a \mapsto \varepsilon(a)$.
Obviously, a dg left (resp. right) $A$-module $M$ over an (augmented) dg $K$-ring is just
a dg right (resp. left) $A^\op$-module over its opposite (augmented) dg $K$-ring, and vice versa.

\bigskip

\noindent{\bf Tensor products of dg $K$-rings.} Let $A$ and $B$ be two dg $K$-rings.
Then the {\it tensor product} $A \otimes_k B$ of $A$ and $B$ is a dg $(K \otimes_k K)$-ring
whose product is $\mu_{A \otimes_k B} : (A \otimes_k B) \otimes_{K \otimes_k K} (A \otimes_k B) \xrightarrow{\cong}
(A \otimes_K A) \otimes_k (B \otimes_K B) \xrightarrow{\mu_A \otimes \mu_B} A \otimes_k B$,
whose unit is $\eta_{A \otimes_k B} : K \otimes_k K \xrightarrow{\eta_A \otimes \eta_B} A \otimes_k B$,
and whose differential is $d_{A \otimes_k B} = d_A \otimes \id + \id \otimes d_B$.
Thus a dg $A$-$B$-bimodule $M$ is nothing but a dg left $A \otimes_k B^\op$-module or a dg right $A^\op \otimes_k B$-module.
If $A$ and $B$ are augmented dg $K$-rings then $A \otimes_k B$ is an augmented dg $(K \otimes_k K)$-ring
with augmentation $\varepsilon_A \otimes \varepsilon_B$.
Let $A$ be a dg $K$-ring. Then $A^e := A^\op \otimes_k A$ is called the {\it enveloping dg $K^e$-ring} of $A$.

\bigskip

\noindent{\bf Dg $K$-corings.} A (counital, coassociative) {\it $K$-coring} $C = (C, \Delta, \varepsilon)$ is a $K$-bimodule $C$ equipped with
two $K$-bimodule morphisms $\Delta: C \rightarrow C \otimes C$, called
{\it coproduct}, and $\varepsilon : C \rightarrow K$, called {\it counit}, satisfying
$(\Delta \otimes \id) \circ \Delta = (\id \otimes \Delta) \circ \Delta$ and
$(\varepsilon \otimes \id) \circ \Delta = \id = (\id \otimes \varepsilon) \circ \Delta$,
i.e., the following two diagrams are commutative:

\medskip

coassociativity \hspace{20mm} $\xymatrix{ C \ar[r]^-{\Delta} \ar[d]_-{\Delta} & C \otimes C \ar[d]^-{\id \otimes \Delta} \\
C \otimes C \ar[r]^-{\Delta \otimes \id} & C \otimes C \otimes C }$

\medskip

counitality \hspace{22mm} $\xymatrix{ C \otimes K \ar@{=}[r] & C \ar[d]^-{\Delta} & K \otimes C. \ar@{=}[l] \\
& C \otimes C \ar[ur]_-{\varepsilon \otimes \id} \ar[ul]^-{\id \otimes \varepsilon} & }$

\medskip

\noindent A {\it morphism} from a $K$-coring $(C, \Delta, \varepsilon)$ to a $K$-coring $(C', \Delta', \varepsilon')$
is a $K$-bimodule morphism $f : C \rightarrow C'$ satisfying $\Delta' \circ f = (f \otimes f) \circ \Delta$
and $\varepsilon' \circ f = \varepsilon$.
A {\it graded $K$-coring} $C = (C, \Delta, \varepsilon)$ is
both a graded $K$-bimodule $C = \bigoplus\limits_{i \in \mathbb{Z}} C_i$ and a $K$-coring $(C, \Delta, \varepsilon)$
such that both $\Delta$ and $\varepsilon$ are morphism of graded $K$-bimodules.
A {\it morphism of graded $K$-corings} is both a morphism of graded $K$-bimodules and a morphism of $K$-corings.
A {\it differential graded (dg) $K$-coring} $C = (C, \Delta, \varepsilon, d)$
is both a dg $K$-bimodule $(C,d)$, whose differential $d : C \rightarrow C$ is of degree $-1$, and a graded $K$-coring $(C, \Delta, \varepsilon)$,
such that both $\Delta$ and $\varepsilon$ are morphisms of dg $K$-bimodules.
A {\it morphism of dg $K$-corings} is both a morphism of dg $K$-bimodules and a morphism of graded $K$-corings.
A {\it coaugmented dg $K$-coring} $C = (C, \Delta, \varepsilon, d, \eta)$ is a dg $K$-coring $(C, \Delta, \varepsilon, d)$
equipped with a dg $K$-coring morphism $\eta : K \rightarrow C$, called {\it coaugmentation}, satisfying $\varepsilon \circ \eta = \id_K$.
If $C$ is a coaugmented dg $K$-coring then $C = K1_C \oplus \overline{C}$ where $1_C = \eta(1_K)$ and
$\overline{C} = \Ker \varepsilon$, called the {\it coaugmentation coideal} of $C$. We always identify $K1_C$ with $K$.
A {\it morphism of coaugmented dg $K$-corings} is a morphism $f : C \rightarrow C'$ of dg $K$-corings
satisfying $f \circ \eta_C = \eta_{C'}$.

\bigskip

\noindent{\bf Dg comodules.} Let $(C, \Delta, \varepsilon)$ be a $K$-coring. A {\it left $C$-comodule} $N = (N, \Delta^l)$
is a left $K$-module $N$ equipped with a left $K$-module morphism $\Delta^l : N \rightarrow C \otimes N$,
called {\it left coaction}, satisfying $(\Delta \otimes \id_N) \circ \Delta^l = (\id_C \otimes \Delta^l) \circ \Delta^l$
and $(\varepsilon \otimes \id_N) \circ \Delta^l = \id_N$, i.e., the following two diagrams are commutative:

\medskip

coassociativity \hspace{10mm} $\xymatrix{ N \ar[r]^-{\Delta^l} \ar[d]_-{\Delta^l} & C \otimes N \ar[d]^{\id_C \otimes \Delta^l} \\
C \otimes N \ar[r]^-{\Delta \otimes \id_N} & C \otimes C \otimes N }$

\medskip

counitality \hspace{16mm} $\xymatrix{ N \ar@{=}[r] \ar[d]_-{\Delta^l} & K \otimes N \\C \otimes N \ar[ur]_-{\varepsilon \otimes \id_N} & }$

\medskip

\noindent Similarly, we can define a {\it right $C$-comodule} with a {\it right coaction} $\Delta^r : N \rightarrow N \otimes C$.
Let $C$ and $D$ be two $K$-corings.
A {\it $C$-$D$-bicomodule} $N = (N, \Delta^l, \Delta^r)$
is not only a $K$-bimodule $N$ but also a left $C$-comodule $(N, \Delta^l)$ and a right $C$-comodule $(N, \Delta^r)$, such that
$(\id_C \otimes \Delta^r) \circ \Delta^l = (\Delta^l \otimes \id_D) \circ \Delta^r.$
A {\it morphism} from a left $C$-comodule $(N, \Delta^l_N)$ to a left $C$-comodule $(N', \Delta^l_{N'})$
is a left $K$-module morphism $f: N \rightarrow N'$ satisfying
$\Delta^l_{N'} \circ f = (\id_C \otimes f) \circ \Delta^l_N.$
Similarly, we can define a {\it morphism of right comodules} and a {\it morphism of bicomodules}.
Let $C$ be a graded $K$-coring.
A {\it graded left $C$-comodule} $N = (N, \Delta^l)$
is both a graded left $K$-module $N = \bigoplus\limits_{i \in \mathbb{Z}} N_i$ and a left $C$-comodule $(N, \Delta^l)$
such that $\Delta^l$ is a morphism of graded left $K$-modules.
A {\it morphism} from a graded left $C$-comodule $N$ to a graded left $C$-comodule $N'$
is both a morphism of graded left $K$-modules and a morphism of left $C$-comodules.
Let $(C, \Delta, \varepsilon, d_C)$ be a dg $K$-coring.
A {\it dg left $C$-comodule} $N = (N, \Delta^l, d_N)$
is both a dg left $K$-module $(N,d_N)$ whose differential $d_N : N \rightarrow N$ is of degree $-1$, and a graded left $C$-comodule $N$,
such that $\Delta^l$ is a morphism of dg left $K$-modules.
A {\it morphism of dg left $C$-comodules} is both a morphism of dg left $K$-modules and a morphism of graded left $C$-comodules.
Similarly, we can define a {\it graded right comodule}, a {\it morphism of graded right comodules},
a {\it graded bicomodule}, a {\it morphism of graded bicomodules}, a {\it dg right comodule},
a {\it morphism of dg right comodules}, a {\it dg bicomodule}, and a {\it morphism of dg bicomodules}.

\bigskip

\noindent{\bf Graded $k$-duals.} Let $(-)^\vee$ be the {\it graded $k$-dual functor} on the category of dg $k$-modules,
i.e., $V^\vee := \bigoplus\limits_{i \in \mathbb{Z}}\Hom_k(V_i, k)$ and $d_{V^\vee} := -d_V^\vee$.
A dg $k$-module $V=(\bigoplus\limits_{i \in \mathbb{Z}}V_i,d)$ is said to be {\it locally finite} or {\it degreewise finite dimensional}
if every component $V_i$ is a finite dimensional $k$-vector space.
Let $V$ be a dg $K$-bimodule. Then there is a dg $K$-bimodule morphism
$\omega_{U,V} : V^\vee \otimes U^\vee \rightarrow (U \otimes V)^\vee$ given by
$\omega_{U,V}(f \otimes g)(u \otimes v) := g(u)f(v)$.
Note that $g(u)=0$ if $|g|+|u| \neq 0$, since $k$ is concentrated on degree 0.
When $U$ and $V$ are locally finite and bounded above (resp. below),
i.e., $U_i=0 =V_i$ for $i \gg 0$ (resp. $U_i=0 =V_i$ for $i \ll 0$), $\omega_{U,V}$ is an isomorphism.
If $(C, \Delta, \varepsilon, d, \eta)$ is a coaugmented dg $K$-coring then
$(C^\vee, \Delta^\vee \circ \omega_{C,C} , \varepsilon^\vee, -d_C^\vee, \eta^\vee)$ is an augmented dg $K$-ring. Here, we identify $K^\vee$ with $K$.
Conversely, if $(A, \mu, \eta, d, \varepsilon)$ is a locally finite bounded above (resp. below) augmented dg $K$-ring,
then $(A^\vee, \omega_{A,A}^{-1} \circ \mu^\vee, \eta^\vee, -d_A^\vee, \varepsilon^\vee)$ is a locally finite bounded below (resp. above) coaugemented dg $K$-coring.

\bigskip

\noindent{\bf (Co)complete dg $K$-(co)rings.}
Let $C$ be a coaugmented dg $K$-coring.
Define $\overline{\Delta} : \overline{C} \rightarrow \overline{C} \otimes \overline C$ by
$\overline{\Delta}(x) = \Delta(x) - 1 \otimes x - x \otimes 1$, and further
$\overline{\Delta}^i : \overline{C} \rightarrow \overline{C}^{\otimes i+1}$
by $\overline{\Delta}^0=\id , \ \overline{\Delta}^1=\overline{\Delta},$ and $\overline{\Delta}^i
= (\overline{\Delta} \otimes \id^{\otimes i-1}) \circ \overline{\Delta}^{i-1}$ for all $i \geq 2$.
Let $F_0C := K$ and $F_iC := K \oplus \Ker \overline{\Delta}^i$ for $i \geq 1$.
Then $F_iC$ is a $C$-bicomodule for all $i \geq 0$. The series
$F_0C \subset \cdots \subset F_iC \subset \cdots$ is called the {\it coradical series} of $C$.
A coaugmented dg $K$-coring $C$ is said to be {\it cocomplete} or {\it conilpotent} if $C = \bigcup\limits_{i=0}^\infty F_iC$.
An augmented dg $K$-ring $A$ is said to be {\it complete} if $\bigcap\limits_{i \geq 1}\overline{A}^i=0$.
Obviously, a finite dimensional augmented dg $K$-ring $A$ is complete if and only if its augmentation ideal $\overline{A}$ is nilpotent,
i.e., there is a positive integer $n$ such that $\overline{A}^n=0$.
A locally finite, bounded above or below, augmented dg $K$-ring $A$ is complete if and only if $A^\vee$ is cocomplete.

\bigskip

\noindent{\bf Tensor dg $K$-rings.} The {\it tensor dg $K$-ring} of a dg $K$-bimodule $V$ is the augmented dg $K$-ring
$$T(V) := \bigoplus\limits_{n=0}^\infty V^{\otimes n} = K \oplus V \oplus V^{\otimes 2} \oplus \cdots$$
whose product is $\mu : T(V) \otimes T(V) \rightarrow T(V), \
(v_1 \otimes \cdots \otimes v_n) \otimes (v_{n+1} \otimes \cdots \otimes v_{n+m}) \mapsto
v_1 \otimes \cdots \otimes v_n \otimes v_{n+1} \otimes \cdots \otimes v_{n+m}$,
whose unit is the embedding $\eta : K \hookrightarrow T(V)$,
whose differential $d$ is given by
$d|_{V^{\otimes n}} :=0$ if $n=0$ and $\sum\limits_{i=1}^n\id^{\otimes i-1}\otimes d_V\otimes \id^{\otimes n-i}$ if $n \geq 1$,
and whose augmentation is the projection $\varepsilon : T(V) \twoheadrightarrow K$,

For a dg $k$-module $M$ and $m \in \mathbb{Z}$, we denote by $Z_mM$ the subspace of $M$ consisting of all $m$-cycles.
Analogous to \cite[Proposition 1.1.1]{LodVal12}, we have the following result:

\begin{proposition} \label{Proposition-Tensor}
Let $V$ be a dg $K$-bimodule and $A$ an augmented dg $K$-ring. Then
$$\Hom_{\rm ADGR}(T(V),A) \cong Z_0\Hom_{K^e}(V,\overline{A}).$$
Here, {\rm ADGR} denotes the category of augmented dg $K$-rings.
\end{proposition}
\begin{proof}
For any $f\in Z_0\Hom_{K^e}(V,\overline{A})$, $f$ is homogeneous and commutes with differentials.
We extend $f$ to $\tilde{f}:T(V)\rightarrow A$ by $\tilde{f}|_K:=\id_K$ and
$\tilde{f}(v_1\otimes\cdots\otimes v_n):=f(v_1)\cdots f(v_n)$ for all $v_1\otimes\cdots\otimes v_n\in V^{\otimes n}$ and $n\geq 1$.
Then $\tilde{f}$ is a morphism of augmented dg $K$-rings.
On the other hand, for any $f\in \Hom_{\rm ADGR}(T(V),A)$,
the composition $V\hookrightarrow T(V)\xrightarrow{f} A\twoheadrightarrow \overline{A}$ belongs to $Z_0\Hom_{K^e}(V,\overline{A})$.
These two constructions are inverse to each other.
\end{proof}

\bigskip

\noindent{\bf Graded $K$-derivations.} Let $A$ be an augmented dg $K$-ring.
A graded $K$-bimodule morphism $d : A \rightarrow A$ of degree $m$ is called a {\it graded $K$-derivation} of $A$ of degree $m$ if

(1) $\varepsilon_A\circ d=0$;

(2) $d_{\Hom_{K^e}(A,A)}(d)=0$, i.e., $d_A\circ d-(-1)^md\circ d_A=0$;

(3) the graded Leibniz rule : $d \circ \mu=\mu \circ (d \otimes \id_A +\id_A \otimes d)$.

In particular, a differential $d:A\rightarrow A$ is a graded $K$-derivation of degree $-1$ such that $d^2=0$.
Denote by $\Der(A)_m$ the $k$-vector space of graded $K$-derivations of $A$ of degree $m$.

Analogous to \cite[2.4. Proposition]{HusMooSta74}, we have the following result:

\begin{proposition} \label{Proposition-Der}
Let $V$ be a dg $K$-bimodule. Then
$$\Der(T(V))_m\cong Z_m\Hom_{K^e}(V,\overline{T(V)}).$$
\end{proposition}

\begin{proof}
For any $f\in Z_m\Hom_{K^e}(V,\overline{T(V)})$, $f$ is a graded $K$-bimodule morphism of degree $m$ such that $d_{\overline{T(V)}}\circ f=(-1)^mf\circ d_V$.
We define $d_f:T(V)\rightarrow T(V)$ by $d_f|_{V^{\otimes n}}:=0$ if $n=0$, and $\sum\limits_{i=1}^n\id^{\otimes i-1}\otimes f\otimes\id^{\otimes n-i}$ if $n \geq 1$.
Then $d_f$ is a graded $K$-derivation of $T(V)$ of degree $m$.
On the other hand, for any $d\in \Der(T(V))_m$, the composition
$V\hookrightarrow T(V)\xrightarrow{d} T(V)\twoheadrightarrow \overline{T(V)}$ belongs to $Z_m\Hom_{K^e}(V,\overline{T(V)})$.
These two constructions are inverse to each other.
\end{proof}

\bigskip

\noindent{\bf Tensor dg $K$-corings.} The {\it tensor dg $K$-coring} of a dg $K$-bimodule $V$ is the coaugmented dg $K$-coring
$$T^c(V) := \bigoplus\limits_{n=0}^\infty V^{\otimes n} = K \oplus V \oplus V^{\otimes 2} \oplus \cdots$$
whose coproduct is $\Delta : T^c(V) \rightarrow T^c(V) \otimes T^c(V), \
v_1 \otimes \cdots \otimes v_n \mapsto \sum\limits_{i=0}^n (v_1 \otimes \cdots \otimes v_i) \otimes (v_{i+1} \otimes \cdots \otimes v_n)$,
for all $v_1 \otimes \cdots \otimes v_n\in V^{\otimes n}$ and $n \geq 1$, and $\Delta(1)=1 \otimes 1$,
whose counit is the projection $\varepsilon : T^c(V) \twoheadrightarrow K$,
whose differential $d$ is given by
$d|_{V^{\otimes n}} := 0$ if $n=0$ and $\sum\limits_{i=1}^n\id^{\otimes i-1} \otimes d_V \otimes \id^{\otimes n-i}$ if $n\geq 1$,
and whose coaugmentation is the embedding $\eta : K \hookrightarrow T^c(V)$,

Analogous to \cite[Proposition 1.2.1]{LodVal12}, we have the following result:

\begin{proposition} \label{Proposition-Tensor2}
Let $V$ be a dg $K$-bimodule and $C$ a cocomplete dg $K$-coring. Then
$$\Hom_{\rm CDGC}(C,T^c(V)) \cong Z_0\Hom_{K^e}(\overline{C},V).$$
Here, {\rm CDGC} denotes the category of coaugmented dg $K$-corings.
\end{proposition}
\begin{proof}
For any $f\in Z_0\Hom_{K^e}(\overline{C},V)$, $f$ is homogeneous and commutes with differentials.
We extend $f$ to $\tilde{f}:C\rightarrow T^c(V)$ by $\tilde{f}|_K:=\id_K$ and
$\tilde{f}|_{\overline{C}}:= \sum\limits_{n=1}^\infty f^{\otimes n} \circ \overline{\Delta}^{n-1}$.
Since $C$ is cocomplete, $\tilde{f}$ is well-defined and it is a morphism of coaugmented dg $K$-corings.
On the other hand, for any $f\in \Hom_{\rm CDGC}(C,T^c(V))$,
the composition $\overline{C}\hookrightarrow C\xrightarrow{f} T^c(V)\twoheadrightarrow V$ belongs to $Z_0\Hom_{K^e}(\overline{C},V)$.
These two constructions are inverse to each other.
\end{proof}

\bigskip

\noindent{\bf Graded $K$-coderivation.} Let $C$ be a cocomplete dg $K$-coring.
A graded $K$-bimodule morphism $d : C \rightarrow C$ of degree $m$ is called a {\it graded $K$-coderivation} of $C$ if

(1) $d\circ\eta_C=0$;

(2) $d_{\Hom_{K^e}(C,C)}(d)=0$, i.e., $d_C\circ d-(-1)^md\circ d_C=0$;

(3) $\Delta \circ d = (d \otimes \id_C) \circ \Delta + (\id_C \otimes d) \circ \Delta$.

In particular, a differential $d:C\rightarrow C$ is a graded $K$-coderivation of degree $-1$ such that $d^2=0$.
Denote by $\Coder(C)_m$ the $k$-vector space of $K$-coderivations of $C$ of degree $m$.

Analogous to \cite[2.6. Proposition]{HusMooSta74}, we have the following result:

\begin{proposition} \label{Proposition-Coder}
Let $V$ be a dg $K$-bimodule. Then
$$\Coder(T^c(V))_m\cong Z_m\Hom_{K^e}(\overline{T^c(V)},V).$$
\end{proposition}

\begin{proof}
For any $f\in Z_m\Hom_{K^e}(\overline{T^c(V)},V)$, $f$ is a graded $K$-bimodule morphism of degree $m$ such that $d_V\circ f=(-1)^mf\circ d_{\overline{T^c(V)}}$.
We define $d_f:T^c(V)\rightarrow T^c(V)$ by $d_f|_K:=0$ and $d_f|_{\overline{T^c(V)}} :=
\sum\limits_{n=1}^\infty\sum\limits_{i=1}^n (\proj_V^{\otimes i-1}\otimes f\otimes\proj_V^{\otimes n-i})\circ \overline{\Delta}^{n-1}$
where $\proj_V : \overline{T^c(V)} \rightarrow V$ is the natural projection.
Then $d_f$ is a graded $K$-coderivation of $T^c(V)$ of degree $m$.
On the other hand, for any $d\in \Coder(T^c(V))_m$,
the composition $\overline{T^c(V)}\hookrightarrow T^c(V)\xrightarrow{d} T^c(V)\twoheadrightarrow V$ belongs to $Z_m\Hom_{K^e}(\overline{T^c(V)},V)$.
These two constructions are inverse to each other.
\end{proof}

\subsection{Twisting morphisms}

Assume that $(A, \mu, \eta_A, d_A, \varepsilon_A)$ is an augmented dg
$K$-ring and $(C, \Delta, \varepsilon_C, d_C, \eta_C)$ is a coaugmented dg $K$-coring.

\bigskip

\noindent{\bf Convolution dg $K$-rings.}
Denote by $\Hom_{K^e}(C, A)$ the $k$-vector space of all graded $K$-bimodule morphisms from $C$ to $A$ of arbitrary degree.
Then $\Hom_{K^e}(C, A)$ is an augmented dg $K$-ring, called the {\it convolution dg $K$-ring} from $C$ to $A$,
whose product is the {\it convolution product}, i.e., $f \star g := \mu \circ (f\otimes g) \circ \Delta$,
whose unit $\eta : K \rightarrow \Hom_{K^e}(C, A)$ is given by $\eta(x) := \eta_A \circ m_x \circ \varepsilon_C$
where $m_x \in \Hom_K(K,K)$ is the left or right multiplication by $x \in K$,
whose differential $d$ is given by $d(f) := d_A \circ f - (-1)^{|f|} f \circ d_C$,
and whose augmentation $\varepsilon : \Hom_{K^e}(C, A) \rightarrow K$ is given by
$\varepsilon(f) := (\varepsilon_A \circ f \circ \eta_C)(1_K)$.

\bigskip

\noindent{\bf Twisting morphisms.}
A {\it twisting morphism} from $C$ to $A$ is a morphism $\alpha \in \Hom_{K^e}(C, A)$ of degree $-1$
satisfying the {\it Maurer-Cartan equation} $d(\alpha) + \alpha \star \alpha=0$,
$\varepsilon_A \circ \alpha = 0$ and $\alpha \circ \eta_C = 0$.
Denote by $\Tw(C, A)$ the set of all twisting morphisms from $C$ to $A$.

For any $\alpha \in \Tw(C,A)$, the {\it twisted differential} $d_{\alpha}$ on
$\Hom_{K^e}(C,A)$ is given by $d_{\alpha}(f) := d(f) + [\alpha, f]$
where $[\alpha, f] := \alpha \star f - (-1)^{|f|}f \star \alpha$.
Then $\Hom_{K^e}^{\alpha}(C, A) := (\Hom_{K^e}(C, A), \star, \eta, d_{\alpha}, \varepsilon)$
is an augmented dg $K$-ring, called the {\it twisting convolution dg $K$-ring}.

For any $\alpha \in \Tw(C, A)$, the {\it twisted differential} $d_{\alpha}$ on $C \otimes A$ is given by
$d_{\alpha} := d_{C\otimes A} + d_{\alpha}^r = d_C \otimes \id_A + \id_C \otimes d_A + d_{\alpha}^r$
where $d_{\alpha}^r := (\id_C \otimes \mu) \circ (\id_C \otimes \alpha \otimes \id_A) \circ (\Delta \otimes \id_A)$.
Then $C \otimes_{\alpha} A := (C \otimes A, d_{\alpha})$ is a dg $K$-bimodule,
called the {\it right twisted tensor product}.
Indeed, $C \otimes_{\alpha} A$ is both a dg left $C$-comodule and a dg right $A$-module.
Similarly, we can define the {\it left twisted tensor product} $A \otimes_{\alpha} C$ with {\it twisted differential}
$d_{\alpha} := d_{A \otimes C} - d_{\alpha}^l$
where $d_{\alpha}^l := (\mu\otimes \id_C) \circ (\id_A \otimes \alpha \otimes \id_C) \circ (\id_A \otimes \Delta)$.

\begin{lemma} \label{Lemma-Hom-Tensor-Twisting}
Let $f: A \rightarrow A'$ be a morphism of augmented dg $K$-rings,
$g: C' \rightarrow C$ a morphism of coaugmented dg $K$-corings, $\alpha \in \Tw(C,A)$.
Then $f \circ \alpha \in \Tw(C,A')$ and $\alpha \circ g \in \Tw(C',A)$. Moreover,

{\rm (1)} the map $f_* : \Hom_{K^e}^{\alpha}(C, A) \rightarrow \Hom_{K^e}^{f \circ \alpha}(C, A'), \ \phi \mapsto f \circ \phi$,
is a morphism of augmented dg $K$-rings;

{\rm (2)} the map $g^* : \Hom_{K^e}^{\alpha}(C, A) \rightarrow \Hom_{K^e}^{\alpha \circ g}(C', A), \ \phi \mapsto \phi \circ g$,
is a morphism of augmented dg $K$-rings;

{\rm (3)} the map $\id \otimes f : C \otimes_{\alpha} A \rightarrow C \otimes_{f \circ \alpha} A'$ is a morphism of dg $K$-bimodules;

{\rm (4)} the map $g \otimes \id : C' \otimes_{\alpha \circ g} A \rightarrow C \otimes_{\alpha} A$ is a morphism of dg $K$-bimodules.
\end{lemma}

\begin{proof} Since $\alpha$ is a twisting morphism, we have $d(\alpha)+\alpha \star \alpha=0$,
$\varepsilon_A \circ \alpha =0$ and $\alpha \circ \eta_C = 0$. Thus
\begin{eqnarray*}
&& d(f \circ \alpha) + (f \circ \alpha)\star(f \circ \alpha) \\
&=& d_{A'} \circ f \circ \alpha - (-1)^{-1} f \circ \alpha \circ d_C + \mu_{A'}\circ((f\circ\alpha)\otimes(f\circ\alpha))\circ\Delta_C \\
&=& f \circ d_A \circ \alpha + f \circ \alpha \circ d_C + \mu_{A'}\circ(f \otimes f)\circ(\alpha\otimes\alpha)\circ\Delta_C \\
&=& f \circ d(\alpha) + f\circ\mu_A\circ(\alpha\otimes\alpha)\circ\Delta_C \\
&=& f \circ d(\alpha) + f \circ (\alpha\star\alpha) \\
&=& f \circ (d(\alpha) + \alpha\star\alpha) \\
&=& 0,
\end{eqnarray*}
$\varepsilon_{A'} \circ (f \circ \alpha) = (\varepsilon_{A'} \circ f) \circ \alpha = \varepsilon_A \circ \alpha =0$
and $(f \circ \alpha) \circ \eta_C = f \circ (\alpha \circ \eta_C) = 0$.
Hence $f \circ \alpha \in \Tw(C,A')$. Similarly, $\alpha \circ g \in \Tw(C',A)$.

(1) It is sufficient to show $f \circ (\phi \star \psi) = (f \circ \phi) \star (f \circ \psi)$ for all $\phi, \psi \in \Hom_{K^e}(C, A)$. Indeed,
\begin{eqnarray*}
f \circ (\phi\star\psi)
&=& f\circ\mu_A\circ(\phi\otimes\psi)\circ\Delta_C\\
&=&\mu_{A'}\circ(f \otimes f)\circ(\phi\otimes\psi)\circ\Delta_C\\
&=&\mu_{A'}\circ((f\circ\phi)\otimes(f\circ\psi))\circ\Delta_C\\
&=&(f\circ\phi)\star(f\circ\psi).
\end{eqnarray*}

(2) Similar to (1).

(3) It is enough to show $d^{r}_{f \circ \alpha}\circ(\id_C\otimes f)=(\id_C\otimes f)\circ d^{r}_{\alpha}$. Indeed,
\begin{eqnarray*}
&&d^{r}_{f \circ \alpha}\circ(\id_C\otimes f)\\
&=&(\id_C\otimes\mu_A)\circ(\id_C\otimes (f \circ \alpha) \otimes\id_A)\circ(\Delta_C\otimes\id_A)\circ(\id_C\otimes f)\\
&=&(\id_C\otimes\mu_A)\circ(\id_C\otimes (f \circ \alpha) \otimes\id_A)\circ(\id_C\otimes\id_C\otimes f)\circ(\Delta_C\otimes\id_A)\\
&=&(\id_C\otimes\mu_A)\circ(\id_C\otimes (f \circ \alpha) \otimes f)\circ(\Delta_C\otimes\id_A)\\
&=&(\id_C\otimes\mu_A)\circ(\id_C\otimes f \otimes f)\circ(\id_C\otimes\alpha\otimes\id_A)\circ(\Delta_C\otimes\id_A)\\
&=&(\id_C\otimes f)\circ(\id_C\otimes\mu_A)\circ(\id_C\otimes\alpha\otimes\id_A)\circ(\Delta_C\otimes\id_A)\\
&=&(\id_C\otimes f)\circ d^{r}_{\alpha}.
\end{eqnarray*}

(4) Similar to (3).
\end{proof}

More generally, let $M$ be a dg left $A$-module, $N$ a dg right $C$-comodule, and $\alpha \in \Tw(C,A)$.
A {\it twisted differential} on $N \otimes M$ is given by
$d_{\alpha} := d_{N \otimes M} + d_{\alpha}^r = d_N \otimes \id_M + \id_N \otimes d_M + d_{\alpha}^r,$
where $d_{\alpha}^r := (\id_N \otimes \mu_M^l) \circ (\id_N \otimes \alpha \otimes \id_M) \circ (\Delta_N^r \otimes \id_M)$.
Then $N \otimes_{\alpha} M := (N \otimes M, d_{\alpha})$ is a dg $k$-module,
called the {\it right twisted tensor product}.
Similarly, let $M$ be a dg right $A$-module, $N$ a dg left $C$-comodule, and $\alpha \in \Tw(C,A)$.
Then the {\it left twisted tensor product} $M \otimes_{\alpha} N := (M \otimes N, d_{\alpha})$ where the {\it twisted differential}
$d_{\alpha} := d_{M \otimes N} - d_{\alpha}^l$
and $d_{\alpha}^l := (\mu_M^r \otimes \id_N) \circ (\id_M \otimes \alpha \otimes \id_N) \circ (\id_M \otimes \Delta_N^l)$.
Let $M$ be a dg right $A$-module, $N$ a dg $C$-bicomodule, $L$ a dg left $A$-module, and $\alpha \in \Tw(C,A)$.
Then the {\it two-sided twisted tensor product} $M \otimes_\alpha N \otimes_\alpha L := (M \otimes N \otimes L, d_{\alpha})$
where the {\it twisted differential}
$d_{\alpha} := d_{M \otimes N \otimes L} + \id_M \otimes d_{\alpha}^r - d_{\alpha}^l \otimes \id_L$,
$d_{\alpha}^r := (\id_N \otimes \mu_L^l) \circ (\id_N \otimes \alpha \otimes \id_L) \circ (\Delta_N^r \otimes \id_L)$
and $d_{\alpha}^l := (\mu_M^r \otimes \id_N) \circ (\id_M \otimes \alpha \otimes \id_N) \circ (\id_M \otimes \Delta_N^l)$.
Similarly, let $N$ be a dg right $C$-comodule, $M$ a dg $A$-bimodule, $L$ a dg left $C$-comodule, and $\alpha \in \Tw(C,A)$.
Then the {\it two-sided twisted tensor product} $N \otimes_\alpha M \otimes_\alpha L := (N \otimes M \otimes L, d_{\alpha})$
where the {\it twisted differential}
$d_{\alpha} := d_{N \otimes M \otimes L} + d_{\alpha}^r \otimes \id_L -\id_N \otimes d_{\alpha}^l$,
$d_{\alpha}^r := (\id_N \otimes \mu_M^l) \circ (\id_N \otimes \alpha \otimes \id_M) \circ (\Delta_N^r \otimes \id_M)$
and $d_{\alpha}^l := (\mu_M^r \otimes \id_L) \circ (\id_M \otimes \alpha \otimes \id_L) \circ (\id_M \otimes \Delta_L^l)$.
Let $M$ be an $A$-bimodule, $N$ a $C$-bicomodule, and $\alpha \in \Tw(C,A)$.
Then the tensor product $M \otimes_{A^e}(A\otimes_\alpha N\otimes_\alpha A) \cong M\otimes_{K^e}N$,
and the induced differential of $M\otimes_{K^e}N$ is $d_\alpha:=d_{M\otimes_{K^e}N}+\tau_{M,N}^{-1} \circ d_\alpha^r \circ \tau_{M,N}-d_\alpha^l$.
Denote the dg $k$-module $(M\otimes_{K^e}N,d_\alpha)$ by $M\otimes^{\alpha}_{K^e}N$.
Similarly, we could define $N\otimes_{K^e}^\alpha M$.
Moreover, the dg $k$-module $\Hom^\alpha_{K^e}(N,M)$ is the graded $k$-vector space $\Hom_{K^e}(N,M)$
endowed with the differential induced from that of $\Hom_{A^e}(A\otimes_{\alpha}N\otimes_{\alpha}A,M)$
under the graded $k$-vector space isomorphism $\Hom_{A^e}(A\otimes_\alpha N\otimes_\alpha A,M)\cong\Hom_{K^e}(N,M)$.

\begin{lemma} \label{Lemma-Hom-Tensor-k-dual}
Let $A$ be a locally finite, bounded above (resp. below), augmented dg $K$-ring, $C$ a locally finite coaugmented dg
$K$-coring, and $\alpha \in \Tw(C,A)$. Then $-\alpha^\vee \in \Tw(A^\vee,C^\vee)$. Moreover,

{\rm (1)} the map $\phi : \Hom_{K^e}^{\alpha}(C,A) \rightarrow \Hom_{K^e}^{-\alpha^\vee}(A^\vee,C^\vee), \ f \mapsto f^\vee,$
is an anti-isomorphism of dg $K$-rings;

{\rm (2)} if $C$ is bounded above (resp. below) then the map $\omega_{C,A} : A^\vee \otimes_{-\alpha^\vee} C^\vee \rightarrow (C \otimes_{\alpha} A)^\vee$
and the map $\omega_{A,C} : C^\vee \otimes_{-\alpha^\vee} A^\vee \rightarrow (A \otimes_{\alpha} C)^\vee$
are isomorphisms of dg $K$-bimodules.

{\rm (3)} if $C$ is bounded above (resp. below) then
the map $\omega_{C,A}: A^\vee\otimes_{K^e}^{-\alpha^\vee}C^\vee \rightarrow (C \otimes_{K^e}^\alpha A)^\vee$
and the map $\omega_{A,C} : C^\vee\otimes_{K^e}^{-\alpha^\vee} A^\vee \rightarrow (A \otimes_{K^e}^\alpha C)^\vee$
are isomorphisms of dg $k$-modules.
\end{lemma}

\begin{proof}
Since $\alpha\in{\rm Tw}(C,A)$, we have $d_A\circ\alpha+\alpha\circ d_C+\alpha\star\alpha=0$, $\varepsilon_A \circ \alpha=0$ and $\alpha \circ \eta_C=0$.
It follows that
$d(-\alpha^\vee)+(-\alpha^\vee)\star(-\alpha^\vee)
= (-\alpha^\vee)\circ d_{A^\vee}+d_{C^\vee}\circ(-\alpha^\vee)+(-\alpha^\vee)\star(-\alpha^\vee)
= \alpha^\vee\circ d_A^\vee+d_C^\vee\circ\alpha^\vee+\alpha^\vee\star\alpha^\vee
= -(d_A\circ\alpha+\alpha\circ d_C+\alpha\star\alpha)^\vee = 0,$
$\varepsilon_{C^\vee} \circ (-\alpha^\vee) = \eta_C^\vee \circ (-\alpha^\vee) = - (\alpha \circ \eta_C)^\vee =0$
and $(-\alpha^\vee) \circ \eta_{A^\vee} = (-\alpha^\vee) \circ \varepsilon_A^\vee = -(\varepsilon_A \circ \alpha)^\vee =0$.
Therefore, $-\alpha^\vee\in{\rm Tw}(A^\vee,C^\vee)$.

(1) Firstly,
$$\phi(f \star g) = (f \star g)^\vee = (-1)^{|f||g|}g^\vee \star f^\vee = (-1)^{|f||g|} \phi(g)\star\phi(f)$$
for all $f,\ g \in \Hom_{K^e}(C,A)$. Secondly, $\phi \circ d = d' \circ \phi$, i.e., $d(f)^\vee=d'(f^\vee)$ for all $f \in \Hom_{K^e}(C,A)$,
where $d := d_{\Hom_{K^e}^{\alpha}(C,A)}$ and $d':=d_{\Hom_{K^e}^{-\alpha^\vee}(A^\vee,C^\vee)}$. Indeed,
\begin{eqnarray*}
d(f)^\vee &=& (d_A \circ f-(-1)^{|f|} f\circ d_C + \alpha\star f-(-1)^{|f|}f\star\alpha)^\vee\\
&=& (-1)^{|f|}f^\vee\circ d_A^\vee-d_C^\vee\circ f^\vee+(-1)^{|f|}f^\vee\star\alpha^\vee-\alpha^\vee\star f^\vee\\
&=& d_{C^\vee}\circ f^\vee-(-1)^{|f|}f^\vee\circ d_{A^\vee}+(-\alpha^\vee)\star f^\vee-(-1)^{|f|}f^\vee\star(- \alpha^\vee)\\
&=& d'(f^\vee).
\end{eqnarray*}
Finally, since $A$ and $C$ are locally finite, $\phi$ is an isomorphism of dg $K$-bimodules.
Therefore, $\phi$ is an anti-isomorphism of dg $K$-rings.

(2) Since both $A$ and $C$ are locally finite and bounded above, $\omega_{C,A}$ is an isomorphism of graded $K$-bimodules.
Thus it is enough to prove $\omega_{C,A}$ commutes with differentials, i.e., $d \circ \omega_{C,A} = \omega_{C,A} \circ d'$,
where $d = d_{(C\otimes_{\alpha} A)^\vee} = -d_{C\otimes_{\alpha} A}^\vee$ and $d' = d_{A^\vee \otimes_{-\alpha^\vee} C^\vee}$. Indeed,
\begin{eqnarray*}
d \circ \omega_{C,A} &=& -(d_C \otimes \id_A + \id_C \otimes d_A + d_{\alpha}^r)^\vee \circ \omega_{C,A} \\
&=& -\omega_{C,A} \circ (\id_A^\vee \otimes d_C^\vee + d_A^\vee \otimes \id_C^\vee - d_{-\alpha^\vee}^r)\\
&=& \omega_{C,A} \circ (\id_{A^\vee} \otimes d_{C^\vee} + {d_A^\vee} \otimes \id_{C^\vee} + d_{-\alpha^\vee}^r)\\
&=& \omega_{C,A} \circ d'.
\end{eqnarray*}

(3) Since both $A$ and $C$ are locally finite and bounded above, $\omega_{C,A}$ is an isomorphism of graded $K$-bimodules.
Thus it suffices to show $\omega_{C,A}$ commutes with differentials. Indeed,
\begin{align*}
& \ d_{(C\otimes_{K^e}^\alpha A)^\vee}\circ \omega_{C,A} \\
=& \ -(d_C\otimes\id_A + \id_C\otimes d_A + d_{\alpha}^r - \tau^{-1}_{C,A}\circ d_\alpha^l\circ\tau_{C,A})^\vee \circ \omega_{C,A}\\
=& \ -((d_C\otimes {\rm id}_A)^\vee+({\rm id}_C\otimes d_A)^\vee+(d_{\alpha}^r)^\vee-\tau^\vee_{C,A}\circ (d_{\alpha}^l)^\vee\circ (\tau_{C,A}^{-1})^\vee)\circ\omega_{C,A}\\
=& \ -\omega_{C,A}\circ(\id_A^\vee\otimes d_C^\vee+d_A^\vee\otimes \id_C^\vee-d_{-\alpha^\vee}^r+\tau^{-1}_{A^\vee,C^\vee}\circ d_{-\alpha^\vee}^l\circ \tau_{A^\vee,C^\vee})\\
=& \ \omega_{C,A}\circ(\id_{A^\vee}\otimes d_{C^\vee}+d_{A^\vee}\otimes \id_{C^\vee}+
d_{-\alpha^\vee}^r-\tau^{-1}_{A^\vee,C^\vee}\circ d_{-\alpha^\vee}^l\circ \tau_{A^\vee,C^\vee})\\
=& \ \omega_{C,A}\circ d_{A^\vee\otimes_{K^e}^{-\alpha^\vee} C^\vee}.
\qedhere
\end{align*}
\end{proof}

\subsection{(Co)bar constructions}

\noindent{\bf Bar construction.} Let $A = K \oplus \overline{A}$ be an augmented dg $K$-ring.
The dg $K$-bimodule $s\overline{A}$, where $s$ is the suspension functor which is also denoted by $[1]$ sometimes,
determines the cocomplete coaugmented dg $K$-coring
$$T^c(s\overline{A}) = \bigoplus\limits_{n=0}^\infty (s\overline{A})^{\otimes n}
= K \oplus s\overline{A} \oplus (s\overline{A})^{\otimes 2} \oplus \cdots$$
whose coproduct $\Delta$, counit $\varepsilon$, differential $d_0$ and coaugmentation $\eta$ have been defined before.
Write $[a_1|a_2|\cdots|a_n]$ for the homogeneous element $sa_1 \otimes sa_2 \otimes \cdots \otimes sa_n \in (s\overline{A})^{\otimes n} \subset T^c(s\overline{A})$
where $s$ in $sa_i$ is the suspension map on $\overline{A}$.
By definition, the differential $d_0$ is given by
$$d_0([a_1|\cdots|a_n]) = \sum\limits_{i=1}^{n} (-1)^{|a_1|+\cdots+|a_{i-1}|+i}[a_1|\cdots|d_A(a_i)|\cdots|a_n].$$
Consider the map $f:\overline{T^c(s\overline{A})}\twoheadrightarrow s\overline{A} \otimes s\overline{A} \xrightarrow{f'} s\overline{A}$,
where the map $f':=s\circ\mu\circ(s^{-1}\otimes s^{-1}) : s\overline{A} \otimes s\overline{A} \rightarrow s\overline{A}, \ [a_1|a_2] \mapsto (-1)^{|a_1|}[a_1a_2]$.
The graded $K$-bimodule morphism $f$ is of degree $-1$ and $d_{s\overline{A}}\circ f+f\circ d_{\overline{T^c(s\overline{A})}}=0$,
i.e., $f\in Z_{-1}\Hom_{K^e}(\overline{T^c(s\overline{A})},s\overline{A})$.
By Proposition \ref{Proposition-Coder}, $f$
induces a graded $K$-coderivation $d_1$ of $T^c(s\overline{A})$ of degree $-1$ given by
$$d_1([a_1|\cdots|a_n]) = \sum\limits_{i=1}^{n-1} (-1)^{|a_1|+\cdots+|a_{i}|+i-1}[a_1|\cdots|a_ia_{i+1}|\cdots|a_n].$$
Thus $d_0\circ d_1+d_1\circ d_0=0$.
Since $d_0^2=0=d_1^2$, $d:=d_0+d_1$ is a differential of $T^c(s\overline{A})$.
The cocomplete coaugmented dg $K$-coring $BA:=(T^c(s\overline{A}),\Delta,\varepsilon,d,\eta)$ is called the {\it bar construction} of $A$.

\bigskip

\noindent{\bf Cobar construction.} Let $C = K \oplus \overline{C}$ be a cocomplete coaugmented dg $K$-coring.
The dg $K$-bimodule $s^{-1}\overline{C}$ determines the augmented dg $K$-ring
$$T(s^{-1}\overline{C}) = \bigoplus\limits_{n=0}^\infty (s^{-1}\overline{C})^{\otimes n}
= K \oplus s^{-1}\overline{C} \oplus (s^{-1}\overline{C})^{\otimes 2} \oplus \cdots $$
whose product $\mu$, unit $\eta$, differential $d_0$ and augmentation $\varepsilon$ have been defined before.
Write $\langle c_1|c_2|\cdots|c_n\rangle$ for the homogeneous element
$s^{-1}c_1 \otimes s^{-1}c_2 \otimes \cdots \otimes s^{-1}c_n \in (s^{-1}\overline{C})^{\otimes n} \subset T(s^{-1}\overline{C}).$
By definition, the differential $d_0$ is given by
$$d_0(\langle c_1|\cdots|c_n \rangle) := \sum\limits_{i=1}^{n} (-1)^{|c_1|+\cdots+|c_{i-1}|+i} \langle c_1|\cdots|d_C(c_i)|\cdots|c_n \rangle.$$
Consider the map $f: s^{-1} \overline{C} \stackrel{f'}{\rightarrow} s^{-1}\overline{C} \otimes s^{-1}\overline{C} \hookrightarrow \overline{T(s^{-1}\overline{C})}$,
where the map $f':=(s^{-1}\otimes s^{-1})\circ\overline{\Delta}\circ s : s^{-1}\overline{C} \rightarrow s^{-1}\overline{C} \otimes s^{-1}\overline{C}, \
\langle c \rangle \mapsto (-1)^{|c_1|+1} \langle c_1|c_2 \rangle$, here $\overline{\Delta}(c) = c_1 \otimes c_2$.
Note that we always omit $\sum$ and brackets in Sweedler's notation $\Delta(c) = \sum c_{(1)} \otimes c_{(2)}$.
The graded $K$-bimodule morphism $f$ is of degree $-1$ and $d_{\overline{T(s^{-1}\overline{C})}}\circ f+f\circ d_{s^{-1}\overline{C}}=0$,
i.e., $f\in Z_{-1}\Hom_{K^e}(s^{-1}\overline{C},\overline{T(s^{-1}\overline{C})})$.
By Proposition \ref{Proposition-Der}, $f$ induces a graded $K$-derivation $d_1$ of $T(s^{-1}\overline{C})$ of degree $-1$ given by
$$d_1(\langle c_1|\cdots|c_n \rangle) = \sum\limits_{i=1}^{n}(-1)^{|c_1|+\cdots+|c_{i-1}|+|c_{i1}|+i} \langle c_1|\cdots|c_{i1}|c_{i2}|\cdots|c_n \rangle$$
where $\Delta(c_i) = c_{i1}\otimes c_{i2}$.
Thus $d_0\circ d_1+d_1\circ d_0=0$. Since $d_0^2=0=d_1^2$, $d:=d_0+d_1$ is a differential of $T(s^{-1}\overline{C})$.
The augmented dg $K$-ring $\Omega C:=(T(s^{-1}\overline{C}),\mu,\eta,d,\varepsilon)$ is called the {\it cobar construction} of $C$.

\bigskip

\noindent{\bf Adjointness.} Analogous to \cite[Theorem 2.2.6]{LodVal12}, we have the following result:

\begin{theorem} \label{Theorem-Cobar-Bar}
Let $A$ be an augmented dg $K$-ring and $C$ a cocomplete dg $K$-coring.
Then there exist natural bijections
$$\Hom_{\rm ADGR}(\Omega C,A) \cong \Tw(C,A) \cong \Hom_{\rm CDGC}(C,BA).$$
In particular, $\Omega$ is left adjoint to $B$.
\end{theorem}

\begin{proof}
For any $\varphi\in\Tw(C,A)$,
$\varphi$ is of degree $-1$ such that $d_A\circ\varphi+\varphi\circ d_C+\varphi\star\varphi=0,\ \varphi\circ\eta_C=0$ and $\varepsilon_A\circ\varphi=0$.
The twisting morphism $\varphi$ induces a graded $K$-bimodule morphism
$\tilde{\varphi}:s^{-1}\overline{C}\stackrel{s}{\cong}\overline{C }\hookrightarrow C\xrightarrow{\varphi}A\twoheadrightarrow\overline{A}$
of degree 0, i.e., $\tilde{\varphi}(s^{-1}c):=\varphi(c)\in\overline{A}$ due to $\varepsilon_A\circ\varphi=0$.
Keep in mind $\Omega C=(T(s^{-1}C),\mu,\eta,d=d_0+d_1,\varepsilon)$,
similar to the proof of Proposition \ref{Proposition-Tensor},
we could show that $\varphi$ induces a unique augmented dg $K$-ring morphism $\Phi:\Omega C\rightarrow A$
given by $\Phi|_K=\id_K$ and $\Phi(\langle c_1|\cdots|c_n\rangle)=\tilde{\varphi}(\langle c_1\rangle)\cdots\tilde{\varphi}(\langle c_1\rangle)$
for all $\langle c_1|\cdots|c_n\rangle \in (s^{-1}\overline{C})^{\otimes n}$ and $n\geq 1$.
Indeed, $d_A\circ\Phi=\Phi\circ(d_0+d_1)$ is equivalent to
$d_A\circ\varphi=-\varphi\circ d_C-\varphi\star\varphi$, thus $\Phi$ commutes with differentials.
The proof of the second bijection is similar.
\end{proof}

\bigskip

\noindent{\bf One-sided (co)bar resolutions.} Applying Theorem \ref{Theorem-Cobar-Bar} to $C = BA$, we get the counit $\epsilon: \Omega BA \rightarrow A$
and a twisting morphism $\pi: BA \rightarrow A$.
Applying Theorem \ref{Theorem-Cobar-Bar} to $A = \Omega C$, we get the unit $\nu: C \rightarrow B\Omega C$
and a twisting morphism $\iota: C \rightarrow \Omega C$,
where $\pi: BA \twoheadrightarrow s\overline{A} \stackrel{s^{-1}}{\cong} \overline{A} \hookrightarrow A$
and $\iota: C \twoheadrightarrow \overline{C} \stackrel{s^{-1}}{\cong} s^{-1}\overline{C} \hookrightarrow \Omega C.$
It follows from Theorem \ref{Theorem-Cobar-Bar} that any twisting morphism $\alpha: C \rightarrow A$
factorizes uniquely through $\pi$ and $\iota$, i.e., we have the following commutative diagram:
$$\xymatrix{ & \Omega C \ar@{.>}[rd]^-{g_\alpha} & \\
C \ar[ur]^\iota \ar[rr]^\alpha \ar@{.>}[dr]_-{f_\alpha} & & A\\ & BA \ar[ur]_\pi & }$$
where $g_\alpha : \Omega C \rightarrow A$ is a morphism of augmented dg $K$-rings
and $f_\alpha : C \rightarrow BA$ is a morphism of coaugmented dg $K$-corings,
which implies that $\pi$ and $\iota$ are {\it universal twisting morphisms}.

\begin{theorem} \label{Theorem-One-Sided-Res}
Let $A$ be an augmented dg $K$-ring and $C$ a cocomplete dg $K$-coring. Then

{\rm (1)} the dg $K$-bimodules $BA \otimes_{\pi} A$ (resp. $A \otimes_{\pi} BA$) is quasi-isomorphic to $K$
as dg right (resp. left) $A$-module, and called the {\rm one-sided bar resolution} of dg right (resp. left) $A$-module $K$.

{\rm (2)} the dg $K$-bimodules $C \otimes_{\iota} \Omega C$ (resp. $\Omega C \otimes_{\iota} C$)
is quasi-isomorphic to $K$ as dg left (resp. right) $C$-comodule, and called the {\rm one-sided cobar resolution} of dg left (resp. right) $C$-comodule $K$.
\end{theorem}

\begin{proof}
Let $p : BA \otimes_{\pi} A \twoheadrightarrow K \otimes K = K$ be the natural projection
and $i : K = K \otimes K \hookrightarrow BA \otimes_{\pi} A$ the natural embedding.
Define a graded $K$-bimodule morphism $s : BA \otimes_{\pi} A \rightarrow BA \otimes_{\pi} A$ of degree 1 by
mapping $[a_1|\cdots|a_n] \otimes a_{n+1}$ to $(-1)^{|a_1|+\cdots+|a_n|+n}[a_1|\cdots|a_n|a_{n+1}] \otimes 1$
if $a_{n+1} \in \overline{A}$, and to 0 if $a_{n+1} \in K$.
Then $p \circ i = \id$ and $\id - i \circ p = d \circ s + s \circ d$,
i.e., $BA \otimes_{\pi} A$ is homotopy equivalent, thus quasi-isomorphic, to $K$.
Similar for other cases.
\end{proof}

\bigskip

\noindent {\bf Quasi-isomorphisms.} The counit $\epsilon$ and the unit $\nu$ of the adjoint pair $(\Omega,B)$ are quasi-isomorphisms.

\begin{theorem} \label{Theorem-Cobar-Bar-Qis}
Let $A$ be an augmented dg $K$-ring and $C$ a cocomplete dg $K$-coring. Then
the counit $\epsilon : \Omega BA \rightarrow A$ is a quasi-isomorphism of dg $K$-rings,
and the unit $\nu : C \rightarrow B\Omega C$ is a quasi-isomorphism of dg $K$-corings.
\end{theorem}

\begin{proof}
Applying Theorem \ref{Theorem-Cobar-Bar} to $C=BA$, we have $\Hom_{\rm ADGR}(\Omega BA,A)\cong \linebreak \Tw(BA,A)$
and $\epsilon$ is completely determined by $\pi_A $.
By the proof of Theorem \ref{Theorem-Cobar-Bar}, we get that $\epsilon: \Omega BA \rightarrow A$ maps
$\langle[a_1]|\cdots|[a_n]\rangle$ to $a_1a_2\cdots a_n$ for all $a_1,a_2,\cdots,a_n\in \overline{A}$, $a$ to $a$ for all $a\in K$, and others to 0.
Define a dg $K$-bimodule morphism $\lambda : A \rightarrow \Omega BA$ by mapping $a$ to $\langle[a]\rangle$ if $a \in \overline{A}$, and $a$ to $a$ if $a \in K$,
and a graded $K$-bimodule morphism $s: \Omega BA \rightarrow \Omega BA$ of degree 1 by
mapping $\langle[a_1]|\cdots|[a_n]|[b_1|\cdots|b_m]|\cdots\rangle$ to
$$\sum\limits_{i=2}^n(-1)^{\beta_{i-1}+1} \langle [a_1\cdots a_{i-1}|a_i]|[a_{i+1}]|\cdots|[a_n]|[b_1|\cdots|b_m]|\cdots\rangle$$
$$+(-1)^{\beta_n+1}\langle [a_1\cdots a_n|b_1|\cdots|b_m]|\cdots\rangle$$
if $n \geq 1$ and $m \geq 2$, $\langle[a_1]|\cdots|[a_n]\rangle$ to
$$\sum\limits_{i=2}^n(-1)^{\beta_{i-1}+1} \langle [a_1\cdots a_{i-1}|a_i]|[a_{i+1}]|\cdots|[a_n]\rangle$$
if $n \geq 2$, and others to 0, where $\beta_i=\sum\limits_{j=1}^i|a_j|, \ a_i, \ b_j \in A,$ for all $1 \leq i \leq n,\ 1 \leq j \leq m$.
Then $\epsilon \circ \lambda = \id$ and $\lambda \circ \epsilon - \id = d \circ s + s \circ d$.
Thus $\epsilon$ is a homotopy equivalence, in particular, a quasi-isomorphism.

Applying Theorem \ref{Theorem-Cobar-Bar} to  $A = \Omega C$, we have  $\Hom_{\rm ADGC}(C,B\Omega C)\cong \Tw(C,\Omega C)$ and  $\nu$ is completely determined by $\iota_C $. By the proof of Theorem \ref{Theorem-Cobar-Bar}, we get that
$\nu: C \rightarrow B\Omega C$ maps $c \in K$ to $c$, and $c \in \overline{C}$ to
$\sum\limits_{i=0}^{\infty} [\langle c_i^1\rangle |\cdots|\langle c_i^{i+1}\rangle]$,
where $\overline{\Delta}^i(c) = c_i^1\otimes\cdots\otimes c_i^{i+1}$.
Define a dg $K$-bimodule morphism $\rho : B\Omega C \rightarrow C$
by mapping $[\langle c\rangle]$ to $c$ for $c \in \overline{C}$, $c$ to $c$ for $c \in K$, and others to 0.
Define a graded $K$-bimodule morphism $h: B\Omega C \rightarrow B\Omega C$ of degree 1 by mapping
$[a_1|\cdots|a_n|\langle c_1|\cdots|c_m\rangle]$ to
$(-1)^\gamma \sum\limits_{i=0}^{\infty} [a_1|\cdots|a_n|\langle c_1|\cdots|c_{m-1}\rangle|\langle
c_{mi}^1 \rangle|\cdots|\langle c_{mi}^{i+1} \rangle]$ if $m \geq 2$ and $n \geq 0$, and to 0 otherwise,
where $a_i \in \Omega C, \ 1 \leq i \leq n, \ c_j \in C, \ 1 \leq j\leq m$,
$\gamma = |a_1|+\cdots+|a_n|+|c_1|+\cdots+|c_{m-1}|+n+m$ and
$\overline{\Delta}^i(c_m) = c_{mi}^1\otimes\cdots\otimes c_{mi}^{i+1}$.
Then $\rho \circ \nu = \id$ and $\nu \circ \rho - \id = d\circ h + h \circ d$.
Thus $\nu$ is a homotopy equivalence, in particular, a quasi-isomorphism.
\end{proof}

\begin{lemma} \label{Lemma-Semiproj} Let $A$ be an augmented dg $K$-ring,
$C$ a cocomplete dg $K$-coring, $\alpha\in \Tw(C,A)$ and $N$ a dg $C$-bicomodule.
Then the dg $A$-bimodule $A\otimes_{\alpha}N\otimes_{\alpha}A$ is semi-projective.
\end{lemma}

\begin{proof}
Define $\overline{\Delta^l} :N\rightarrow \overline{C}\otimes N$ by $\overline{\Delta^l}(x)=\Delta^l(x)-1\otimes x$,
where $\Delta^l$ is the left coaction on $N$.
Similarly, define $\overline{\Delta^r}: N\rightarrow N\otimes \overline{C}$ by $\overline{\Delta^r}(x)=\Delta^r(x)-x\otimes 1$,
where $\Delta^r$ is the right coaction on $N$.
Then $(\id\otimes \overline{\Delta^l})\circ \overline{\Delta^l}=(\overline{\Delta}\otimes \id)\circ\overline{\Delta^l}$
and $(\overline{\Delta^r}\otimes \id)\circ \overline{\Delta^r}=(\id\otimes \overline{\Delta})\circ\overline{\Delta^r}$.
Define $(\overline{\Delta^l})^n: N\rightarrow\overline{C}^{\otimes n}\otimes N$ for $n\geq 1$ inductively by
$(\overline{\Delta^l})^1:=\overline{\Delta^l}$ and
$(\overline{\Delta^l})^n:=(\id\otimes(\overline{\Delta^l})^{n-1})\circ\overline{\Delta^l}$ for $n\geq 2$.
Define $(\overline{\Delta^r})^n: N\rightarrow N\otimes \overline{C}^{\otimes n}$ for $n\geq 1$ inductively by
$(\overline{\Delta^r})^1:=\overline{\Delta^r}$ and $(\overline{\Delta^r})^n:=((\overline{\Delta^r})^{n-1}\otimes\id)\circ\overline{\Delta^r}$ for $n\geq 2$.
Then
$(\overline{\Delta}^{n-1}\otimes \id)\circ\overline{\Delta^l}=(\overline{\Delta^l})^n=(\id\otimes(\overline{\Delta^l})^{n-1})\circ\overline{\Delta^l}$
and $(\id\otimes\overline{\Delta}^{n-1})\circ\overline{\Delta^r}=(\overline{\Delta^r})^n=((\overline{\Delta^r})^{n-1}\otimes\id)\circ\overline{\Delta^r}$.

Put $N_{ij}=\Ker(\overline{\Delta^l})^i\cap\Ker(\overline{\Delta^r})^j$ for all $i,j\geq 1$. Then
$N_{ij}\subset N_{i+1,j}$, $N_{ij}\subset N_{i,j+1}$, $\overline{\Delta^l}(N_{ij})\subset \overline{C}\otimes N_{i-1,j}$
and $\overline{\Delta^r}(N_{ij})\subset N_{i,j-1}\otimes\overline{C}$.
Let $F_m=\sum\limits_{i+j=m}N_{ij}$. Then $N$ admits a dg $C$-bicomodule filtration
$$0=F_1 \subset F_2 \subset \cdots \subset F_{m-1} \subset F_m \subset \cdots \subset N.$$
Since $C$ is cocomplete, for any $c\in \overline{C}$, there exists a positive integer $n$ such that $\overline{\Delta}^n(c)=0$.
Hence, for any $x\in N,\ (\overline{\Delta^l})^{n+1}(x)=((\overline{\Delta}^{n}\otimes \id)\circ\overline{\Delta^l})(x)
=(\overline{\Delta}^n\otimes \id)(c_1\otimes x_2)=0$ for some large enough positive integer $n$.
Similarly, $(\overline{\Delta^r})^m(x)=0$ for some large enough positive integer $m$.
Thus $x$ belongs to $N_{n,m}$.
Therefore, the filtration is exhaustive, i.e., $N=\bigcup\limits_{m=1}^\infty F_m$.

Since $d_NF_m \subset F_m$ for all $m \geq 1$, we have a refined dg $C$-bicomodule filtration
$$0=F_1 \subset \cdots \subset F_{m-1} \subset F_{m-1}+d_NF_m \subset F_m \subset F_m+d_NF_{m+1} \subset F_{m+1} \subset \cdots \subset N.$$
Such a filtration of $N$ induces an exhaustive  dg $A$-bimodule filtration of $A \otimes_\alpha N \otimes_\alpha A$:
$$0=A\otimes_\alpha F_1 \otimes_\alpha A \subset \cdots \subset A\otimes_\alpha F_m \otimes_\alpha A \subset
A\otimes_\alpha(F_m + d_NF_{m+1})\otimes_\alpha A \subset \cdots \subset A \otimes_\alpha N \otimes_\alpha A.$$
The differential of $A \otimes_\alpha N\otimes_\alpha A$ is
$$d_{A} \otimes \id \otimes \id + \id \otimes d_N \otimes \id + \id \otimes \id \otimes d_{A} + \id \otimes d_{\alpha}^r - d_{\alpha}^l \otimes \id.$$
Denote $\id \otimes d_N \otimes \id + \id \otimes d_{\alpha}^r - d_{\alpha}^l \otimes \id$ by $\tilde{d}$.
Then
$$\tilde{d}(A\otimes_\alpha F_m \otimes_\alpha A) \subset A\otimes_\alpha(F_{m-1} + d_NF_m)\otimes_\alpha A$$
and
$$\tilde{d}(A\otimes_\alpha(F_m + d_NF_{m+1})\otimes_\alpha A) \subset A\otimes_\alpha F_m \otimes_\alpha A.$$
Thus the differentials of the dg $A$-bimodule subfactors
$$(A\otimes_\alpha F_{m+1} \otimes_\alpha A)/(A\otimes_\alpha (F_m+d_NF_{m+1}) \otimes_\alpha A) \cong A\otimes_\alpha (F_{m+1}/(F_m+d_NF_{m+1})) \otimes_\alpha A$$
and
$$(A\otimes_\alpha (F_m + d_NF_{m+1}) \otimes_\alpha A)/(A\otimes_\alpha F_m \otimes_\alpha A) \cong A\otimes_\alpha((F_m + d_NF_{m+1})/F_m) \otimes_\alpha A$$
are just $d_A \otimes \id \otimes \id + \id \otimes \id \otimes d_A$,
i.e., all these subfactors are relatively projective dg $A$-bimodules \cite{Kel94}.
Therefore, $A \otimes_\alpha N \otimes_\alpha A$ is semi-projective.
\end{proof}

\bigskip

\noindent{\bf Two-sided bar resolutions.} For an augmented dg $K$-ring $A$,
there is a standard semi-projective resolution of the dg $A$-bimodule $A$, i.e., its two-sided bar resolution.

\begin{theorem} \label{Theorem-Two-sidedBarRes}
Let $A$ be an augmented dg $K$-ring.
Then $A \otimes_{\pi} BA \otimes_{\pi} A$ is a semi-projective resolution of the dg $A$-bimodule $A$,
called the {\rm two-sided bar resolution} of $A$. More precisely, $A \otimes_{\pi} BA \otimes_{\pi} A$ is a semi-projective dg $A$-bimodule and the composition
$$\tilde{\mu} : A \otimes_\pi BA \otimes_\pi A \xrightarrow{\id \otimes \varepsilon \otimes \id} A \otimes K \otimes A = A \otimes A \xrightarrow{\mu} A$$
is a quasi-isomorphism of dg $A$-bimodules.
\end{theorem}

\begin{proof}
Since $BA$ is cocomplete, by Lemma \ref{Lemma-Semiproj}, the dg $A$-bimodule $A\otimes_\pi BA\otimes_\pi A$ is semi-projective.
To prove that $\tilde{\mu}$ is a quasi-isomorphism, it is enough to show that $\tilde{\mu}$ is a homotopy equivalence.
For this, we define a homotopy inverse $i : A \rightarrow A \otimes_\pi BA \otimes_\pi A, \
a \mapsto 1\otimes 1\otimes a$, of $\tilde{\mu}$, and a graded $K$-bimodule morphism
$s: A \otimes_\pi BA \otimes_\pi A \rightarrow A \otimes_\pi BA \otimes_\pi A$ of degree 1
by mapping $a_0\otimes[a_1|\cdots|a_n]\otimes a_{n+1}$ to $1\otimes[a_0|a_1|\cdots|a_n]\otimes a_{n+1}$.
Then $\tilde{\mu} \circ i = \id_A$ and $i \circ \tilde{\mu}-\id = d \circ s + s \circ d$.
Thus $\tilde{\mu}$ is a homotopy equivalence.
\end{proof}

\begin{remark}{\rm
The two-sided bar resolution $A \otimes_\pi  B A \otimes_\pi A$ of an augmented dg $K$-ring $A$
is a semi-projective resolution of the dg $A$-bimodule $A$,
which is smaller than the usual two-sided bar resolution of the augmented dg algebra $A$.
Use this smaller semi-projective resolution, some calculations on Hochschild (co)homology and cyclic homology become simpler \cite{Cil90}.
}\end{remark}

\section{Hochschild (co)homologies of dg $K$-rings and their Koszul duals}

In this section, we will formulate the calculus theory of the Hochschild homology and cohomology of augmented dg $K$-rings,
and compare the Hochschild (co)homologies of a complete typical dg $K$-ring $A$ and its Koszul dual $\Omega A^\vee$.
Moreover, we will study the homological smoothness of $\Omega A^\vee$.

\subsection{Calculi}

\noindent {\bf Hochschild cochain complex.} The {\it Hochschild cochain complex} of $A$ is
$C^\bullet(A) := \Hom_{K^e}^\pi(BA,A) \cong \Hom_{A^e}(A \otimes_\pi BA \otimes_\pi A, A) \cong \RHom_{A^e}(A,A)$.
Its cohomology is called the {\it Hochschild cohomology} of $A$, denoted by $HH^\bullet(A)$.
Note that the bullet $\bullet$ indicates weight in $C^\bullet(A)$ and degree in $HH^\bullet(A)$.
More precisely, $C^\bullet(A) = \prod\limits_{n \in \mathbb{N}} C^n(A)$ where $C^n(A) := \Hom_{K^e}((s\overline{A})^{\otimes n},A)$,
and the differential of $C^\bullet(A)$ is $\dz = \dz_0+\dz_1$ where $\dz_0$ is the {\it inner differential} given by
$$\dz_0(f)[a_1|\cdots|a_n] := d_Af[a_1|\cdots|a_n] + \sum_{i=1}^n(-1)^{\varepsilon_{i-1}+|f|}f[a_1|\cdots|d_Aa_i|\cdots|a_n],$$
and $\dz_1$ is the {\it external differential} given by
$$\begin{array}{ll} \dz_1(f)[a_1|\cdots|a_{n+1}] := & (-1)^{|f|(|a_1|+1)}a_1f[a_2|\cdots|a_{n+1}] \\
& \quad + \sum\limits_{i=1}^n(-1)^{|f|+\varepsilon_i}f[a_1|\cdots|a_ia_{i+1}|\cdots|a_{n+1}] \\
& \quad\quad + (-1)^{|f|+\varepsilon_n+1}f[a_1|\cdots|a_n]a_{n+1}, \end{array}$$
for all $f \in C^n(A)$.
Here $\varepsilon_i := \sum\limits_{j=1}^{i}(|a_j|+1)$.

\begin{remark}{\rm
(1) The definition of $\delta$, more precisely, $\delta_1$, here is somewhat different from that in \cite{Men09,Abb15,Her15}.
Indeed, ours is induced from the $K$-reduced two-sided bar resolution.

(2) Compare with the calculus theory in other literatures, we use $K$-reduced two-sided bar resolution here
rather than the ordinary two-sided bar resolution as a semi-projective resolution of the dg $A$-bimodule $A$.
Thus our Hochschild cochains are different, more precisely, less. However, the related results and their proofs are same.
Just because of this, we only look though the whole calculus theory but omit almost all proofs.
}\end{remark}

\bigskip

\noindent {\bf Cup product.} For $f \in C^m(A)$ and $g \in C^n(A)$, the {\it cup product} $f \cup g \in C^{m+n}(A)$ is defined by
$$(f \cup g)[a_1|\cdots|a_{m+n}] := (-1)^{|g|\varepsilon_m}f[a_1|\cdots|a_m]g[a_{m+1}|\cdots|a_{m+n}].$$
where $\varepsilon_m:=\sum\limits_{j=1}^m(|a_j|+1)$ as before.

\begin{remark}{\rm
Since $BA$ is a coaugmented dg $K$-coring and $A$ is an augmented dg $K$-ring,
$\Hom_{K^e}^\pi(BA,A)$ is an augmented dg $K$-ring on convolution product.
The cup product is just the convolution product.
}\end{remark}

The cup product satisfies the following properties:

\begin{lemma} \label{Lemma-Cup} For all $f \in C^m(A)$ and $g \in C^n(A)$, one has

{\rm (1)} $\dz(f \cup g) = \dz f \cup g + (-1)^{|f|}f \cup \dz g$;

{\rm (2)} $(f \cup g) \cup h = f \cup (g \cup h)$;

{\rm (3)} $f \cup g -(-1)^{|f||g|}g \cup f = (-1)^{|f|+1} \dz f \circ g + (-1)^{|f|}\dz(f \circ g) + f \circ \dz g$.
\end{lemma}

By Lemma \ref{Lemma-Cup} (1), the cup product on $C^\bullet(A)$ can induce the cup product on $HH^\bullet(A)$.
It follows from Lemma \ref{Lemma-Cup} (2) and (3) that $(HH^\bullet(A),\cup)$ is a graded commutative associative algebra.

\bigskip

\noindent {\bf Gerstenhaber bracket.} The {\it Gerstenhaber bracket} $[f,g] \in C^{m+n-1}(A)$ is defined by
$$[f,g] := f \circ g - (-1)^{(|f|+1)(|g|+1)}g \circ f,$$
where the product $\circ$ is defined by
$$f \circ g := \sum_{i=1}^{m} (-1)^{\varepsilon_{i-1}(|g|+1)}f[a_1|\cdots|g[a_i|\cdots|a_{i+n-1}]|\cdots|a_{m+n-1}].$$

The Gerstenhaber bracket $[-,-]$ satisfies the following properties:

\begin{lemma} \label{Lemma-GerBra} For all $f \in C^m(A)$, $g \in C^n(A)$ and $h \in C^p(A)$, one has

{\rm (1)} $\dz[f,g] = [\dz f,g]+(-1)^{|f|+1}[f,\dz g]$;

{\rm (2)} $[f,g] = -(-1)^{(|f|+1)(|g|+1)}[g,f]$;

{\rm (3)} {\rm (Graded Jacobi identity):}
$$(-1)^{(|f|+1)(|h|+1)}[[f,g],h]+(-1)^{(|g|+1)(|f|+1)}[[g,h],f]+(-1)^{(|h|+1)(|g|+1)}[[h,f],g]=0.$$
\end{lemma}

By Lemma \ref{Lemma-GerBra} (1), the Gerstenhaber bracket on $C^\bullet(A)$ can induce the Gerstenhaber bracket on $HH^\bullet(A)$.
It follows from Lemma \ref{Lemma-GerBra} (2) and (3) that $(HH^{\bullet+1}(A),[-,-])$ is a graded Lie algebra.

\bigskip

\noindent {\bf Graded Leibniz rule.} The cup product and the Gerstenhaber bracket are related by the graded Leibniz rule:

\begin{lemma} \label{Lemma-LeibnizRule}
For all $f \in C^m(A)$, $g \in C^n(A)$ and $h \in C^p(A)$, one has
$$[f,g \cup h] = [f,g] \cup h + (-1)^{(|f|+1)|g|}g \cup [f,h].$$
\end{lemma}

\begin{definition}{\rm A {\it Gerstenhaber algebra} $A^\bullet = (A^\bullet, \cdot , [-,-])$ is
both a graded commutative associative algebra $(A^\bullet, \cdot)$ and a graded Lie algebra $(A^{\bullet+1},[-,-])$
such that the graded Leibniz rule holds, i.e.,
$$[a,b \cdot c] = [a,b] \cdot c + (-1)^{(|a|+1)|b|}b \cdot [a,c]$$
for all $a,b,c \in A^\bullet$.
}\end{definition}

\begin{theorem} {\rm (Gerstenhaber \cite{Ger63})} Let $A$ be an augmented dg $K$-ring. Then $(HH^{\bullet}(A), \cup, [-,-])$ is a Gerstenhaber algebra.
\end{theorem}

\bigskip

\noindent {\bf Hochschild chain complex.} The {\it Hochschild chain complex} of $A$ is
$C_\bullet(A) := A\otimes^\pi_{K^e}BA \cong A \otimes_{A^e}(A \otimes_\pi BA \otimes_\pi A) \cong A \otimes^L_{A^e} A$.
Its homology is called the {\it Hochschild homology} of $A$, denoted by $HH_\bullet(A)$.
Note that the bullet $\bullet$ indicates weight in $C_\bullet(A)$ and degree in $HH_\bullet(A)$.
More precisely, $C_\bullet(A) = \bigoplus\limits_{n \in \mathbb{N}} C_n(A)$ with $C_n(A) := A \otimes (s\overline{A})^{\otimes n}$,
and the differential of $C_\bullet(A)$ is $b := b_0 + b_1$ where $b_0$ is the {\it inner differential} given by
$$b_0(a_0 \otimes [a_1|\cdots|a_n]) := d_Aa_0 \otimes [a_1|\cdots|a_n] + \sum_{i=1}^n(-1)^{\eta_{i-1}}a_0 \otimes [a_1|\cdots|d_Aa_j|\cdots|a_n],$$
and $b_1$ is the {\it external differential} given by:
$$\begin{array}{ll} b_1(a_0 \otimes [a_1|\cdots|a_n]) := & (-1)^{|a_0|+1}a_0a_1 \otimes [a_2|\cdots|a_n] \\
& \quad + \sum\limits_{i=1}^{n-1}(-1)^{\eta_i}a_0 \otimes [a_1|\cdots|a_ia_{i+1}|\cdots|a_n] \\
& \quad\quad + (-1)^{(\eta_{n-1}+1)(|a_n|+1)}a_na_0 \otimes [a_1|\cdots|a_{n-1}] \end{array}$$
where $\eta_i := \sum\limits_{j=0}^i(|a_j|+1)$.

\bigskip

\noindent {\bf Connes operator.} The {\it Connes operator} $B$ on $C_\bullet(A)$ is defined by
$$B(a_0 \otimes [a_1|\cdots|a_n]) := \sum_{i=0}^n (-1)^{\eta_i(\eta_n-\eta_i)} 1 \otimes [a_{i+1}|\cdots|a_n|a_0|\cdots|a_i].$$
It satisfies $B^2 =0$ and $Bb+bB=0$. Thus it can induce a differential of $HH_\bullet(A)$ of degree $1$.

\bigskip

\noindent {\bf Operator $i_f$.} For any $f \in C^m(A)$, the cap product \cite{CarEil56} induces an operator $i_f$ on $C_\bullet(A)$ given by
$$i_f(a_0 \otimes [a_1|\cdots|a_n]) := (-1)^{|a_0||f|}a_0f[a_1|\cdots|a_m] \otimes [a_{m+1}|\cdots|a_n].$$

The operator $i_f$ satisfies the following properties:

\begin{lemma} \label{Lemma-i_f} For all $f \in C^m(A)$ and $g \in C^n(A)$, one has

{\rm (1)} $[b,i_f] = i_{\dz f}$;

{\rm (2)} $i_{f \cup g} = (-1)^{|f||g|}i_g \circ i_f.$
\end{lemma}

Lemma \ref{Lemma-i_f} implies that if $f$ is a Hochschild cocycle then the operator $i_f$ can induce an operator $i_f$ on $HH_\bullet(A)$.
Moreover, the map $i : HH^\bullet(A) \rightarrow \End(HH_\bullet(A)), \ f \mapsto i_f,$
is an (anti-)morphism of graded algebras, i.e., $HH_{\bullet}(A)$ is a graded (right/left) module over the graded commutative associative algebra $(HH^{\bullet}(A),\cup)$.

\bigskip

\noindent {\bf Operator $L_f$.} To $f \in C^m(A)$, we associate an operator $L_f$ on $C_\bullet(A)$ given by
$$\begin{array}{l}
L_f(a_0 \otimes [a_1|\cdots|a_n]):= \\
\quad \sum\limits_{i=1}^{n-m}(-1)^{(\eta_i+1)(|f|+1)} a_0 \otimes [a_1|\cdots|a_i|f[a_{i+1}|\cdots|a_{i+m}]|a_{i+m+1}]|\cdots|a_n]+ \\
\sum\limits_{i=n-m+1}^n(-1)^{(\eta_n-\eta_i)\eta_i+|f|+1} f[a_{i+1}|\cdots|a_n|a_0|\cdots|a_{m+i-n-1}] \otimes [a_{m+i-n}|\cdots|a_i].
\end{array}$$

The operator $L_f$ satisfies the following properties:

\begin{lemma} \label{Lemma-L_f} For all $f \in C^m(A)$ and $g \in C^n(A)$, one has

{\rm (1)} $[b, L_f] + L_{\dz f} = 0$;

{\rm (2)} $L_{[f,g]} = [L_f, L_g]$.
\end{lemma}

Lemma \ref{Lemma-L_f} implies that if $f$ is a Hochschild cocycle then the operator $L_f$ can induce an operator $L_f$ on $HH_\bullet(A)$.
Moreover, the map $L : HH^{\bullet+1}(A) \rightarrow \End(HH_\bullet(A)), \ f \mapsto L_f$,
is a morphism of graded Lie algebras, i.e., $HH_{\bullet}(A)$ is a graded left module over the graded Lie algebra $(HH^{\bullet+1}(A),[-,-])$.

\bigskip

\noindent {\bf Operator $S_f$.} To $f \in C^m(A)$, we associate an operator $S_f$ on $C_\bullet(A)$ given by
$$\begin{array}{l}
S_f(a_0 \otimes [a_1|\cdots|a_p]) := \\
\quad\quad \sum\limits_{i=m+1}^{p+1}\sum\limits_{j=1}^{i-m} (-1)^{\epsilon_{ij}} 1 \otimes [a_i|\cdots|a_p|a_0|\cdots|f[a_j|\cdots|a_{j+m-1}]|\cdots|a_{i-1}],
\end{array}$$
where $\epsilon_{ij} := (|f|+1)(\eta_p+\eta_{i-1}-\eta_{j-1}+1)+\eta_{i-1}(\eta_p-\eta_{i-1})$.

\bigskip

The operators $i_f, \ L_f$ and $S_f$ are related by the Cartan formula:

\begin{lemma} \label{Lemma-S_f} For all $f \in C^m(A)$, one has the {\rm Cartan formula}
$$L_f = [B,i_f]-[b,S_f]+S_{\delta f}.$$
More precisely, $L_f = [B,i_f]-[b_1,S_f]+S_{\delta_1 f}$ and $[b_0,S_f] = S_{\delta_0 f}$.
\end{lemma}

\bigskip

\noindent {\bf Operator $T_{(f,g)}$.} To $f \in C^m(A)$ and $g \in C^n(A)$, we associate an operator $T_{(f,g)}$ on $C_\bullet(A)$ given by
$$\begin{array}{l} T_{(f,g)}(a_0 \otimes [a_1|\cdots|a_p]) := \sum\limits_{i=p-m+3}^{p+1}\sum_{j=1}\limits^{m+i-p-2}(-1)^{\eta_{i,j}} \\ f[a_i|\cdots|a_p|a_0|\cdots|g[a_j|\cdots|a_{j+n-1}]|\cdots|a_{n+m+i-p-3}] \otimes [a_{n+m+i-p-2}|\cdots|a_{i-1}]
\end{array}$$
where $\eta_{i,j} := (\eta_p-\eta_{i-1})\eta_{i-1}+(|g|+1)(\eta_p-\eta_{i-1}+\eta_{i-1}+1)$
and $a_0$ must be located between $f$ and $g$.

\begin{lemma} \label{Lemma-T_{(f,g)}} For all $f \in C^m(A)$ and $g \in C^n(A)$, one has

{\rm (1)} $[L_f,i_g] + (-1)^{|f|}i_{[f,g]} = -[b, T_{(f,g)}] + T_{(\delta f,g)} + (-1)^{|f|}T_{(f,\delta g)}$;

{\rm (2)} $[B,T_{(f,g)}] = [L_f,S_g] + (-1)^{|f|}S_{[f,g]}$;

{\rm (3)} $[b_0,T_{(f,g)}] = T_{(\dz_0f,g)} + (-1)^{|f|}T_{(f,\dz_0g)}$.
\end{lemma}

The operators $i_f$, $L_f$ and the Connes operator $B$ on $HH_\bullet(A)$ satisfy the following properties:

\begin{lemma} \label{Lemma-i_f-L_f} For all $f,g \in HH^\bullet(A)$, one has

{\rm (1)} $[i_f,i_g]=0$;

{\rm (2)} $i_{f \cup g} = i_f \circ i_g$;

{\rm (3)} $L_{[f,g]} = [L_f, L_g]$;

{\rm (4)} $L_f = [B,i_f]$;

{\rm (5)} $i_{[f,g]} = [i_f,L_g]$;

{\rm (6)} $L_{f \cup g} = L_f \circ i_g + (-1)^{|f|}i_f \circ L_g$.
\end{lemma}

\begin{proof}
(1) and (2) follow from Lemma \ref{Lemma-i_f}.

(3) is obtained from Lemma \ref{Lemma-L_f}.

(4) follows from Lemma \ref{Lemma-S_f}.

(5) is obained from Lemma \ref{Lemma-T_{(f,g)}}.

(6) By (4) and (2), we have
$L_f \circ i_g + (-1)^{|f|}i_f \circ L_g = [B,i_f] \circ i_g + (-1)^{|f|}i_f \circ [B,i_g]
= (B \circ i_f -(-1)^{|f|}i_f \circ B) \circ i_g + (-1)^{|f|}i_f \circ (B \circ i_g -(-1)^{|g|}i_g \circ B )
= [B, i_{f \cup g}] = L_{f \cup g}$.
\end{proof}

\begin{definition} {\rm (Tamarkin-Tsygan \cite{TamTsy00,TamTsy05})
A {\it calculus} is a triple $(G^\bullet, M^\bullet,d)$ where

(1) $G^\bullet = (G^\bullet, \cdot , [-,-])$ is a Gerstenhaber algebra;

(2) $M^{-\bullet}$ is a graded module over the graded commutative associative algebra $G^\bullet$,
and the action of $a \in G^\bullet$ on $M^{-\bullet}$ is denoted by $i_a$;

(3) $M^{-\bullet}$ is a graded left module over the graded Lie algebra $G^{\bullet+1}$,
and the action of $a \in G^{\bullet+1}$ on $M^{-\bullet}$ is denoted by $L_a$;

(4) $d$ is a differential on $M^{\bullet}$ of degree 1;

such that $$i_{[a,b]}=[i_a, L_b], \ \ \ L_{a \cdot b} = L_a \circ i_b + (-1)^{|a|}i_a \circ L_b , \ \ \ L_a = [d,i_a].$$
}\end{definition}

By Lemma \ref{Lemma-i_f-L_f}, we have the following result:

\begin{theorem} Let $A$ be an augmented dg $K$-ring. Then $(HH^\bullet(A), HH_\bullet(A), B)$ is a calculus. \end{theorem}

We will see that it is exactly the calculi that determine the BV algebra structures
on the Hochschild cohomologies of $d$-CY dg algebras and $d$-symmetric dg $K$-rings.

\subsection{Comparison}

Assume that $A$ is a locally finite, bounded above or below, complete dg $K$-ring,
which ensures that $A^\vee$ is a locally finite cocomplete dg $K$-coring.
In this part, we will study the relations between the Hochschild (co)homologies of $A$ and its Koszul dual.
Here, the {\it Koszul dual} of a locally finite, bounded above or below, augmented dg $K$-ring $A$ is defined to be $\Omega A^\vee$ (See \cite{VdB15,Kel03}).
In some literatures, it is defined to be the augmented dg $K$-ring $(BA)^\vee =\Hom_k(BA,k)$ (See \cite{Kel94}),
whose product is $\mu : (BA)^\vee \otimes_K (BA)^\vee \rightarrow (BA \otimes_KBA)^\vee \xrightarrow{\Delta^\vee} (BA)^\vee$,
whose unit is $\varepsilon^\vee$, whose differential is $-d^\vee$, and whose augmentation is $\eta^\vee$,
where $BA=(BA, \Delta, \varepsilon, d, \eta)$ is the bar construction of $A$.
In practice, $\Omega A^\vee$ is more convenient to handle than $(BA)^\vee$.
Later on, we will see that if $A$ is a typical dg $K$-ring then $(BA)^\vee \cong \Omega A^\vee$.

\bigskip

\noindent{\bf Typical dg $K$-rings.} A {\it typical dg $K$-ring} is a locally finite augmented dg $K$-ring $A$ which is either non-negative or
non-positive simply connected (i.e., $A_0=K$ and $A_{-1}=0$). The typicality of a dg $K$-ring $A$
ensures that both the bar construction $BA$ and $(BA)^\vee$ are locally finite.

Let $A$ be a typical dg $K$-ring. For any $n \in \mathbb{N}$, we define a dg $K$-bimodule morphism
$$\psi_n : (s^{-1}\overline{A^\vee})^{\otimes n} \rightarrow ((s\overline{A})^{\otimes n})^\vee, \
\langle f_1|\cdots|f_n \rangle \mapsto \psi_n\langle f_1|\cdots|f_n \rangle$$
where $\psi_n\langle f_1|\cdots|f_n \rangle \in ((s\overline{A})^{\otimes n})^\vee$ is given by
$$(\psi_n\langle f_1|\cdots|f_n \rangle)[a_1|\cdots|a_n] := (-1)^{\varepsilon_n}f_n(a_1)f_{n-1}(a_2)\cdots f_1(a_n)$$
with $\varepsilon_n := \sum\limits_{i=1}^n(|f_i|+1)$. Since $BA$ is locally finite,
$(s^{-1}\overline{A^\vee})^{\otimes n} \cong ((s\overline{A})^{\otimes n})^\vee$ as dg $K$-bimodules. Thus
$$\psi := \bigoplus\limits_{n \in \mathbb{N}} \psi_n : \ \Omega A^\vee \rightarrow (BA)^\vee$$
is a dg $K$-bimodule isomorphism.

In fact, it is routine to check the following result:

\begin{proposition} \label{Proposition-KoszulDuals}
Let $A$ be a typical dg $K$-ring. Then $\psi : \Omega A^\vee \rightarrow (BA)^\vee$ is an isomorphism of dg $K$-rings.
\end{proposition}

\bigskip

{\bf From now on, we often denote $\Omega A^\vee$ by $\Omega$ for simplicity.}

\bigskip

\noindent{\bf A semi-projective resolution of $\Omega A^\vee$.}
Recall that the bar construction of the augmented dg $K$-ring $\Omega A^\vee$ is the coaugmented dg $K$-coring $B\Omega A^\vee = T^c(s\overline{\Omega A^\vee})$,
and the map $\pi : B \Omega A^\vee \twoheadrightarrow s\overline{\Omega A^\vee} \xrightarrow{s^{-1}} \overline{\Omega A^\vee}
\hookrightarrow \Omega A^\vee$ is the universal twisting morphism.
By Theorem \ref{Theorem-Two-sidedBarRes}, the two-sided bar resolution $\Omega \otimes_\pi B \Omega A^\vee \otimes_\pi \Omega$ of $\Omega$
is a semi-projective resolution of the dg $\Omega$-bimodule $\Omega$.
Due to the assumption at the beginning of this part, the dg $K$-coring $A^\vee$ is cocomplete.
On the other hand, by Theorem \ref{Theorem-Cobar-Bar-Qis}, we have a quasi-isomorphism of dg $K$-corings
$$\nu: A^\vee \rightarrow B \Omega A^\vee, \ x \mapsto \sum\limits_{i=0}^{\infty}[\overline{\Delta}^i(x)]$$
where $\overline{\Delta}^i(x) = x_i^1\otimes\cdots\otimes x_i^{i+1}$
and $[\overline{\Delta}^i(x)] := [\langle x_i^1 \rangle |\cdots|\langle x_i^{i+1} \rangle]$ for all $x \in \overline{A^\vee}=\overline{A}^\vee$,
and $\nu|_K=\id_K$.
The composition
$$\iota := \pi \circ \nu \ : \ A^\vee \rightarrow B \Omega A^\vee \rightarrow \Omega A^\vee$$
maps $x$ to $\langle x \rangle \in \Omega A^\vee$ if $x \in \overline{A^\vee}$, and to 0 if $x \in K$.
Obviously, $\iota$ is a twisting morphism,
and the twisted tensor product $\Omega \otimes_\iota A^\vee \otimes_\iota \Omega$ is a dg $\Omega$-bimodule.

\begin{proposition} \label{Proposition-Semi-proj-Res-Omega}
Let $A$ be a locally finite, bounded above or below, complete dg $K$-ring.
Then the dg $\Omega$-bimodule $\Omega \otimes_{\iota} A^\vee \otimes_{\iota} \Omega$ is a semi-projective resolution of the dg $\Omega$-bimodule $\Omega$.
\end{proposition}

\begin{proof}
By assumption, the dg $K$-coring $A^\vee$ is cocomplete.
It follows from Lemma \ref{Lemma-Semiproj} that $\Omega \otimes_\iota A^\vee \otimes_{\iota} \Omega$ is semi-projective.

Next, we show that $\Omega \otimes_{\iota} A^\vee \otimes_{\iota} \Omega$ is quasi-isomorphic to $\Omega$ as dg $\Omega$-bimodules.
The quasi-isomorphism $\nu: A^\vee \rightarrow B \Omega A^\vee$ induces a dg $\Omega$-bimodule morphism
$$\tilde{\nu} := \id \otimes \nu \otimes \id : \ \Omega \otimes_{\iota} A^\vee \otimes_{\iota} \Omega \rightarrow \Omega \otimes_{\pi} B \Omega A^\vee \otimes_{\pi} \Omega.$$
Define a dg $\Omega$-bimodule morphism
$$\tilde{\rho} : \Omega \otimes_{\pi} B \Omega A^\vee \otimes_{\pi} \Omega \rightarrow \Omega \otimes_{\iota} A^\vee \otimes_{\iota} \Omega$$ by
mapping $x=u \otimes [\langle v_1|\cdots|v_n \rangle] \otimes w$ to
$u \cdot (\sum\limits_{i=1}^{n} (-1)^{\varepsilon_{i-1}} \langle v_1|\cdots|v_{i-1} \rangle \otimes v_i \otimes \langle v_{i+1}|\cdots|v_n \rangle) \cdot w$
if $n \geq 1$, to $x$ if $n=0$, and mapping others to 0. Here $\varepsilon_{i-1} = \sum\limits_{j=1}^{i-1}(|v_j|+1)$.
Now it is enough to prove that the morphism $\tilde{\nu}$ is a homotopy equivalence with homotopy inverse $\tilde{\rho}$.
For this, let $s : \Omega \otimes_{\pi} B \Omega A^\vee \otimes_{\pi} \Omega \rightarrow \Omega \otimes_{\pi} B \Omega A^\vee \otimes_{\pi} \Omega$
be the graded $\Omega$-bimodule morphism of degree 1 defined by mapping
$a_0 \otimes [a_1|\cdots|a_{n-1}|\langle v_1|\cdots|v_m \rangle] \otimes a_n \in \Omega \otimes B \Omega A^\vee \otimes \Omega$ to
$$\sum_{i=2}^m \sum_{j=0}^\infty (-1)^{\zeta_{n-1,i-1}}a_0\otimes[a_1|\cdots|a_{n-1}|\langle v_1|\cdots|v_{i-1} \rangle |
\langle v_{ij}^1 \rangle |\cdots|\langle v_{ij}^{j+1} \rangle] \otimes \langle v_{i+1}|\cdots|v_m \rangle a_n$$
if $m \geq 2$, and to 0 otherwise,
where $\zeta_{n-1,i-1} = \sum\limits_{p=0}^{n-1}(|a_p|+1)+\sum\limits_{q=1}^{i-1}(|v_q|+1)$
and $\overline{\Delta}^j(v_i) = v_{ij}^1 \otimes \cdots \otimes v_{ij}^{j+1}$ (cf. \cite[3.2.6]{Her15}).
Then $\tilde{\rho} \circ \tilde{\nu} = \id$ and $\tilde{\nu} \circ \tilde{\rho} - \id = d \circ s + s \circ d$.
\end{proof}

\begin{remark}{\rm
The completeness of $A$ is crucial in the proof of Proposition \ref{Proposition-Semi-proj-Res-Omega}.
Let $A$ be a finite dimensional elementary algebra,
or equivalently, a bound quiver algebra $A=kQ/I$ with $Q$ a finite quiver and $I$ an admissible ideal of $kQ$.
Then $A$ is a complete $kQ_0$-ring with nilpotent augmentation ideal $(Q_1)$.
Certainly, $A$ is also an augmented algebra with the augmentation given by some projection
$A \twoheadrightarrow kQ_0 \twoheadrightarrow k$.
However, in this case, $A$ is not complete, or equivalently, the augmentation ideal is not nilpotent, in general.
This is why we consider dg $K$-(co)rings rather than dg algebras.
}\end{remark}

\bigskip

\noindent{\bf Homological smoothness of $\Omega A^\vee$.} A dg algebra $A$ is {\it homologically smooth} if $A \in \per A^e$ (See \cite{KonSoi09,Gin06}).
Here, $\per A^e$ is the perfect derived category of $A^e$,
i.e., the full triangulated subcategory of the derived category $\mathcal{D}A^e$ of the dg algebra $A^e$ consisting of all perfect (=compact) objects,
or equivalently, the smallest full triangulated subcategory of $\mathcal{D}A^e$ containing $A^e$ and closed under direct summands \cite{Kel94}.

The following result generalizes \cite[Proposition 5.6]{Lun10}.

\begin{theorem} \label{Theorem-Homologically-Smooth}
Let $A$ be a finite dimensional complete dg $K$-ring.
Then $\Omega A^\vee$ of $A$ is a homologically smooth dg algebra.
\end{theorem}

\begin{proof}
By Proposition \ref{Proposition-Semi-proj-Res-Omega},
$\Omega \otimes_{\iota} A^\vee \otimes_{\iota} \Omega$ is a semi-projective resolution of the dg $\Omega$-bimodule $\Omega$.
Since $A$ is finite dimensional, so is $A^\vee$.
Thus the filtration of $\Omega \otimes_\iota A^\vee \otimes_\iota \Omega$ given by the proof of Lemma \ref{Lemma-Semiproj} is finite.
Moreover, every dg $\Omega$-bimodule subfactor is a direct summand of a finite direct sum of shifts of $\Omega^e$, thus in $\per \Omega^e$.
Hence the dg $\Omega$-bimodule $\Omega \otimes_{\iota} A^\vee \otimes_{\iota} \Omega$ lies in $\per \Omega^e$,
and further, $\Omega$ lies in $\per \Omega^e$ as well.
\end{proof}

\bigskip

\noindent{\bf Comparison of Hochschild (co)homologies.} Now we compare
the Hochschild (co)homologies of a complete typical dg $K$-ring $A$ and its Koszul dual $\Omega = \Omega A^\vee$.

\begin{theorem} \label{Theorem-HHHH_Connes}
Let $A$ be a complete typical dg $K$-ring. Then

{\rm (1)} there is an isomorphism $h^\bullet : HH^\bullet(A) \rightarrow HH^\bullet(\Omega A^\vee)$ of graded commutative associative algebras;

{\rm (2)} there is an isomorphism
$h_\bullet : HH_{\bullet}(A)^\vee \rightarrow HH_{-\bullet}(\Omega A^\vee)$ of graded $k$-modules
such that the following diagram is commutative:
$$\xymatrix{ HH_{\bullet}(A)^\vee \ar[r]^-{B^\vee} \ar[d]^-{h_\bullet} & HH_{\bullet}(A)^\vee \ar[d]^-{h_\bullet}\\
HH_{-\bullet}(\Omega A^\vee) \ar[r]^-{B} & HH_{-\bullet}(\Omega A^\vee). }$$
\end{theorem}

\begin{proof}

(1) Since $A$ is locally finite, from Lemma \ref{Lemma-Hom-Tensor-k-dual} (1), we obtain an anti-isomorphism of dg $K$-rings
$$\phi : \Hom_{K^e}^{\pi}(BA,A) \rightarrow \Hom_{K^e}^{-\pi^\vee}(A^\vee,(BA)^\vee), \ f \mapsto f^\vee.$$
Since $A$ is typical, by Proposition \ref{Proposition-KoszulDuals}, we have a dg $K$-ring isomorphism $\psi : \Omega A^\vee \rightarrow (BA)^\vee$.
Clearly, $\psi \circ \iota = - \pi^\vee$. Thus it is a twisting morphism.
It follows from Lemma \ref{Lemma-Hom-Tensor-Twisting} (1) that
$$\psi_* = \Hom(A^\vee,\psi) \ : \ \Hom_{K^e}^{\iota}(A^\vee,\Omega A^\vee) \rightarrow \Hom_{K^e}^{-\pi^\vee}(A^\vee,(BA)^\vee)$$
is a dg $K$-ring isomorphism.
Consequently, the composition
$$\psi_*^{-1} \circ \phi \ : \ \Hom^{\pi}_{K^e}(BA,A) \xrightarrow{\phi} \Hom_{K^e}^{-\pi^\vee}(A^\vee,(BA)^\vee) \xrightarrow{\psi_*^{-1}}
\Hom_{K^e}^{\iota}(A^\vee,\Omega)$$
is an anti-isomorphism of dg $K$-rings.

Moreover, the homotopy equivalences $\tilde{\nu}$ and $\tilde{\rho}$ in the proof of Proposition \ref{Proposition-Semi-proj-Res-Omega}
induce the homotopy equivalences $\widetilde{\tilde{\nu}^*}$ and $\widetilde{\tilde{\rho}^*}$ by the following diagram
$$\xymatrix@!=4pc{ \Hom_{K^e}^{\iota}(A^\vee,\Omega) \ar@<+0.5ex>@{.>}[rrr]^-{\widetilde{\tilde{\rho}^*}} \ar[d]^-{\cong} &&&
\Hom_{K^e}^\pi(B\Omega A^\vee,\Omega) \ar@<+0.5ex>@{.>}[lll]^-{\widetilde{\tilde{\nu}^*}} \ar[d]^-{\cong} \\
\Hom_{\Omega^e}(\Omega\otimes_\iota A^\vee \otimes_\iota\Omega, \Omega) \ar@<+0.5ex>[rrr]^-{\tilde{\rho}^*} &&&
\Hom_{\Omega^e}(\Omega\otimes_\pi B\Omega A^\vee \otimes_\pi\Omega, \Omega). \ar@<+0.5ex>[lll]^-{\tilde{\nu}^*} }$$
By Lemma \ref{Lemma-Hom-Tensor-Twisting}(2), $\widetilde{\tilde{\nu}^*}$ preserves convolution products.
Then the composition
$$\theta^{-1} :=\phi^{-1}\circ\psi_*\circ\widetilde{\tilde{\nu}^*} :\  \Hom_{K^e}^\pi(B\Omega A^\vee,\Omega) \rightarrow \Hom^{\pi}_{K^e}(BA,A)$$
is a homotopy equivalence preserving convolution products with homotopy inverse $\theta:=\widetilde{\tilde{\rho}^*}\circ \psi_*^{-1}\circ \phi$.

Since cup products coincide with convolution products, by taking cohomology,
we obtain an (anti-)isomorphism of graded commutative associative algebras
$$h^\bullet := H^\bullet(\theta) \ : \ HH^\bullet(A) \rightarrow HH^\bullet(\Omega A^\vee).$$

(2) The dg $K$-ring isomorphism $\psi : \Omega A^\vee \rightarrow (BA)^\vee$ induces a graded $k$-vector space isomorphism
$$\psi\otimes\id \ : \ \Omega A^\vee \otimes^{\iota}_{K^e} A^\vee \rightarrow (BA)^\vee \otimes^{-\pi^\vee}_{K^e} A^\vee.$$
Note that the differential of $\Omega A^\vee \otimes^\iota_{K^e} A^\vee$ is
$d_\iota := d_{\Omega} \otimes \id + \id \otimes d_{A^\vee} + \tau^{-1}_{\Omega ,A^\vee} \circ d_\iota^r \circ \tau_{\Omega ,A^\vee} - d_\iota^l$
and the differential of $(BA)^\vee \otimes^{-\pi^\vee}_{K^e} A^\vee$ is
$d_{-\pi^\vee} := d_{(BA)^\vee} \otimes \id + \id \otimes d_{A^\vee} + \tau^{-1}_{(BA)^\vee,A^\vee} \circ d_{-\pi^\vee}^r \circ \tau_{(BA)^\vee,A^\vee} - d_{-\pi^\vee}^l$.
It is easy to check that $\psi \otimes \id : \Omega \otimes^\iota_{K^e} A^\vee \rightarrow (BA)^\vee\otimes^{-\pi^\vee}_{K^e} A^\vee$
commutes with differentials, and thus is a dg $k$-module isomorphism.

By Lemma \ref{Lemma-Hom-Tensor-k-dual} (3), we have a dg $k$-vector space isomorphism
$$\omega_{A,BA} : (BA)^\vee \otimes^{-\pi^\vee}_{K^e} A^\vee \rightarrow (A \otimes^\pi_{K^e} BA)^\vee.$$

Moreover, the homotopy equivalences $\tilde{\nu}$ and $\tilde{\rho}$ in the proof of Proposition \ref{Proposition-Semi-proj-Res-Omega}
induce the homotopy equivalences $\tilde{\tilde{\nu}}$ and $\tilde{\tilde{\rho}}$ by the following diagram
$$\xymatrix@!=4pc{
\Omega \otimes_{K^e}^{\iota} A^\vee \ar@<+0.5ex>@{.>}[rrr]^-{\tilde{\tilde{\nu}}} \ar[d]^-{\cong} &&&
\Omega\otimes_{K^e}^\pi B\Omega A^\vee \ar@<+0.5ex>@{.>}[lll]^-{\tilde{\tilde{\rho}}} \ar[d]^-{\cong} \\
\Omega\otimes_{\Omega^e}(\Omega\otimes_\iota A^\vee \otimes_\iota\Omega) \ar@<+0.5ex>[rrr]^-{\id \otimes \tilde{\nu}} &&&
\Omega\otimes_{\Omega^e}(\Omega\otimes_\pi B\Omega A^\vee \otimes_\pi\Omega). \ar@<+0.5ex>[lll]^-{\id \otimes \tilde{\rho}} }$$

Consequently, we have a homotopy equivalence
$$\vartheta := \tilde{\tilde{\nu}} \circ (\psi \otimes \id)^{-1} \circ \omega_{A,BA}^{-1} \ : \ (A \otimes^\pi_{K^e} BA)^\vee \rightarrow \Omega\otimes_{K^e}^\pi B\Omega A^\vee$$
with homotopy inverse $\omega_{A,BA} \circ (\psi \otimes \id) \circ \tilde{\tilde{\rho}}$.
Furthermore, the Connes operators $B$ on $A \otimes^\pi_{K^e} BA$ and $\Omega\otimes_{K^e}^\pi B\Omega A^\vee$ satisfy the following commutative diagram
$$\xymatrix{ (A \otimes^\pi_{K^e} BA)^\vee \ar[d]^-{\vartheta} \ar[r]^-{B^\vee} & (A \otimes^\pi_{K^e} BA)^\vee\\
\Omega\otimes_{K^e}^\pi B\Omega A^\vee \ar[r]^-{B} & \Omega\otimes_{K^e}^\pi B\Omega A^\vee \ar[u]_-{\omega_{A,BA} \circ (\psi \otimes \id) \circ \tilde{\tilde{\rho}}}. }$$

By taking homologies, we obtain a graded $k$-module isomorphism
$$h_\bullet := H_\bullet(\vartheta) \ : \ HH_{\bullet}(A)^\vee \rightarrow HH_{-\bullet}(\Omega A^\vee)$$
and the following diagram is commutative:
$$\xymatrix{ HH_{\bullet}(A)^\vee \ar[r]^-{B^\vee} \ar[d]^-{h_\bullet} & HH_{\bullet}(A)^\vee \ar[d]^-{h_\bullet}\\
HH_{-\bullet}(\Omega A^\vee) \ar[r]^-{B} & HH_{-\bullet}(\Omega A^\vee). }$$
\end{proof}

\section{Symmetric dg $K$-rings and Calabi-Yau dg algebras}

In this section, we show that the Koszul dual $\Omega A^\vee$ of a finite dimensional complete typical $d$-symmetric dg $K$-ring $A$
is a $d$-CY dg algebra, its Hochschild cohomology $HH^\bullet(\Omega A^\vee)$ is a BV algebra,
and $HH^\bullet(\Omega A^\vee)$ and $HH^\bullet(A)$ are isomorphic as BV algebras.

\subsection{Koszul duals of symmetric dg $K$-rings}

{\bf Symmetric dg $K$-rings.} Elementary symmetric algebra could be generalized to symmetric dg $K$-ring on three levels.

\begin{definition}{\rm
An augmented dg $K$-ring $A$ is said to be

{\it derived $d$-symmetric} if there is an isomorphism $A \cong A^\vee[-d]$ in $\mathcal{D}A^e$;

{\it $d$-symmetric} if there is a dg $A$-bimodule quasi-isomorphism $A\rightarrow A^\vee[-d]$;

{\it strictly $d$-symmetric} if there is a dg $A$-bimodule isomorphism $A\cong A^\vee[-d]$.
}\end{definition}

\bigskip

\noindent{\bf A dg bimodule isomorphism.} In order to show that the Koszul dual of a finite dimensional complete typical
$d$-symmetric dg $K$-ring is $d$-CY,
we need a dg bimodule isomorphism. Let $A$ be an augmented dg $K$-ring, $C$ a finite dimensional coaugmented dg $K$-coring,
and $\pi \in \Tw(C,A)$.
Then the graded $k$-dual $C^\vee$ of $C$ is an augmented dg $K$-ring. Its product defines a dg $C^\vee$-bimodule structure on $C^\vee$,
which induces a dg $C$-bicomodule structure on $C^\vee$.
The left coaction $\Delta^l$ is the image of $\Delta^\vee$ under the composition
\begin{align*}
\Hom_K({_K}(C \otimes C)^\vee,{_K}C^\vee) & \cong \Hom_K({_K}C^\vee \otimes C^\vee,{_K}C^\vee) \\ & \cong \Hom_K({_K}C^\vee, \Hom_K({_K}C^\vee,{_K}C^\vee)) \\
& \cong \Hom_K({_K}C^\vee, {_K}C^{\vee\vee} \otimes C^\vee) \\ & \cong \Hom_K({_K}C^\vee, {_K}C \otimes C^\vee)
\end{align*}
and the right coaction $\Delta^r$ is the image of $\Delta^\vee$ under the composition
\begin{align*}
\Hom_K((C \otimes C)^\vee_K,C^\vee_K) & \cong \Hom_K(C^\vee \otimes C^\vee_K,C^\vee_K) \\ & \cong \Hom_K(C^\vee_K, \Hom_K(C^\vee,C^\vee)_K) \\
& \cong \Hom_K(C^\vee_K, C^\vee \otimes C^{\vee\vee}_K) \\ & \cong \Hom_K(C^\vee_K, C^\vee \otimes C_K).
\end{align*}
Recall that $A \otimes_{\pi} C^\vee \otimes_{\pi} A$ is the dg $A$-bimodule $A\otimes C^\vee \otimes A$ with the differential
$d_{A\otimes C^\vee\otimes A} + \id\otimes d_{\pi}^r -d_{\pi}^l\otimes\id$ where $d_{\pi}^r$ is the composition
$$C^\vee\otimes A \xrightarrow{\Delta^r\otimes\id} C^\vee\otimes C \otimes A
\xrightarrow{\id\otimes\pi\otimes\id} C^\vee\otimes A\otimes A \xrightarrow{\id\otimes\mu} C^\vee\otimes A$$
and $d_{\pi}^l$ is the composition
$$A\otimes C^\vee \xrightarrow{\id\otimes \Delta^{l}} A\otimes C\otimes C^\vee
\xrightarrow{\id\otimes\pi\otimes\id} A\otimes A\otimes C^\vee \xrightarrow{\mu\otimes\id} A\otimes C^\vee.$$
Thus $\Hom_{A^e}(A \otimes_\pi C \otimes_\pi A,A^e)$ is a dg $A^e$-module with the $A^e$-module structure induced from
the inner action of $A^e$ on $A^e$, i.e., $a \otimes b$ acts from the left on $x \otimes y$
by $(a \otimes b)(x \otimes y) := (-1)^{|a||b|+|a||x|+|b||x|}xb \otimes ay$.
Its differential $d$ is induced from that of $A \otimes_\pi C \otimes_\pi A.$
Therefore, for all $h \in \Hom_{A^e}(A \otimes_\pi C \otimes_\pi A,A^e)$ and $1 \otimes c \otimes 1 \in A \otimes_\pi C \otimes_\pi A$,
\begin{align*}
d(h)(1 \otimes c \otimes 1) = & d_{A^e}(h(1 \otimes c \otimes 1))-(-1)^{|h|}h(1 \otimes d_C(c) \otimes 1) \\
& \quad -(-1)^{|h|+|c_1|}h(1 \otimes c_1 \otimes 1)\pi(c_2) \\
& \quad \quad +(-1)^{|c_1||h|}\pi(c_1)h(1 \otimes c_2 \otimes 1)
\end{align*}
where $\Delta(c)=c_1\otimes c_2$.

\begin{lemma} \label{Lemma-BimodIso}
Let $A$ be an augmented dg $K$-ring, $C$ a finite dimensional coaugmented dg $K$-coring, and $\pi \in \Tw(C,A)$.
Then the map $\Phi : A \otimes_{\pi}C^\vee \otimes_{\pi} A \rightarrow \Hom_{A^e}(A \otimes_\pi C \otimes_\pi A,A^e),
\ a \otimes f \otimes b \mapsto \Phi(a\otimes f\otimes b)$,
given by $\Phi(a\otimes f\otimes b)(1 \otimes c \otimes 1) := (-1)^{(|a|+|f|)|b|}f(c)b \otimes a$, is an isomorphism of dg $A$-bimodules.
\end{lemma}

\begin{proof}
Taking the canonical $k$-basis $e_1, \cdots , e_t$ of $K=k^t$ which is a complete set of orthogonal primitive idempotent elements of the $k$-algebra $K$,
the map $\Phi$ is just the graded $A$-bimodule isomorphism
\begin{align*}
A\otimes_\pi C^\vee\otimes_\pi A & \cong \bigoplus_{1 \leq i,j \leq t}Ae_j\otimes_k(e_iCe_j)^\vee\otimes_ke_iA \\
& \cong \bigoplus_{1 \leq i,j \leq t}\Hom_k(e_iCe_j,e_iA\otimes_kAe_j) \\
& \cong \Hom_{K^e}^\pi(C,A^e) \\
& \cong \Hom_{A^e}(A \otimes_\pi C \otimes_\pi A,A^e).
\end{align*}
Now it is sufficient to show that $\Phi$ commutes with differentials, i.e.,
$\Phi \circ d = d \circ \Phi$. Since $\Phi$ is a graded $A$-bimodule morphism, by graded Leibniz rule, it is enough to prove that
$(\Phi\circ d)(1\otimes f\otimes 1) = (d\circ \Phi)(1\otimes f\otimes 1)$ for all $f \in C^\vee$, or equivalently,
$(\Phi\circ d)(1\otimes f\otimes 1)(1 \otimes c \otimes 1) = (d\circ \Phi)(1\otimes f\otimes 1)(1 \otimes c \otimes 1)$ for all $f \in C^\vee$ and $c \in C$.

Note that $d(1\otimes f \otimes 1) = 1\otimes d_{C^\vee}(f)\otimes 1 + (\id\otimes d_{\pi}^r)(1\otimes f\otimes 1)
-(d_{\pi}^l\otimes \id)(1\otimes f \otimes 1)$. Suppose $\Delta^r(f) = f' \otimes c'$. Then
\begin{align*}
& \ (\id\otimes d_{\pi}^r)(1\otimes f\otimes 1) \\
= & \ (\id\otimes\id \otimes \mu)(\id\otimes \id\otimes \pi\otimes \id)(\id\otimes \Delta^r\otimes \id)(1\otimes f\otimes 1) \\
= & \ (\id\otimes\id \otimes \mu)(\id\otimes \id\otimes \pi\otimes \id)(1\otimes f'\otimes c'\otimes 1)\\
= & \ (\id\otimes\id \otimes \mu)((-1)^{|f'|}  1\otimes f' \otimes \pi(c')\otimes 1) \\
= & \ (-1)^{|f'|}1\otimes f'\otimes \pi(c').
\end{align*}
Similarly, suppose $\Delta^l(f) = c'' \otimes f''$. Then $(d_{\pi}^l\otimes\id)(1\otimes f\otimes 1) = \pi(c'')\otimes f'' \otimes 1$.
Thus $(\Phi\circ d)(1\otimes f\otimes 1) = \Phi(1\otimes d_{C^\vee}(f)\otimes 1+(-1)^{|f'|}1\otimes f' \otimes \pi(c')-\pi(c'')\otimes f'' \otimes 1)$.
Suppose $\Delta(c) = c_1\otimes c_2$. Then
\begin{align*}
& \ (\Phi\circ d)(1\otimes f\otimes 1)(1 \otimes c \otimes 1) \\
= & \ \Phi(1\otimes d_{C^\vee}(f)\otimes 1+(-1)^{|f'|}1\otimes f' \otimes \pi(c')-\pi(c'')\otimes f'' \otimes 1)(1 \otimes c \otimes 1) \\
= & \ -(-1)^{|f|}f(d(c))1\otimes 1 +(-1)^{|f'||c'|}f'(c)\pi(c')\otimes 1-f''(c)1\otimes\pi(c'')\\
= & \ -(-1)^{|f|}f(d(c))1\otimes 1 +(-1)^{|f'||c'|}\pi(f'(c)c')\otimes 1-1\otimes\pi(f''(c)c'')
\end{align*}
and
\begin{align*}
& \ (d\circ \Phi)(1\otimes f\otimes 1)(1 \otimes c \otimes 1) \\
= & \ d_{A^e}(\Phi(1\otimes f\otimes 1)(1 \otimes c \otimes 1))-(-1)^{|f|}\Phi(1\otimes f\otimes 1)(1 \otimes d(c) \otimes 1) \\
& \ \quad - (-1)^{|f|+|c_1|}\Phi(1\otimes f\otimes 1)(1 \otimes c_1 \otimes 1)\pi(c_2) \\
& \ \quad \quad +(-1)^{|c_1||f|}\pi(c_1)\Phi(1\otimes f\otimes 1)(1 \otimes c_2 \otimes 1) \\
= & \ -(-1)^{|f|}f(d(c))1\otimes 1 -f(c_1)1\otimes \pi(c_2) +(-1)^{|c_1||f|}\pi(c_1)f(c_2)\otimes 1\\
= & \ -(-1)^{|f|}f(d(c))1\otimes 1 -1\otimes\pi(f(c_1)c_2) +(-1)^{|c_1||f|}\pi(c_1f(c_2))\otimes 1.
\end{align*}

Note that $\Delta^r(f) = f' \otimes c'$ and $\Delta(c) = c_1\otimes c_2$.
From the definition of $\Delta^r$, we can obtain $(-1)^{|f'||c'|}f'(c)c'=(-1)^{|c_1||f|}c_1f(c_2)$.
Similarly, from $\Delta^l(f) = c'' \otimes f''$ and $\Delta(c) = c_1\otimes c_2$, we can obtain $f''(c)c''=f(c_1)c_2$.
Therefore, $(\Phi\circ d)(1\otimes f\otimes 1)(1 \otimes c \otimes 1) = (d\circ \Phi)(1\otimes f\otimes 1)(1 \otimes c \otimes 1)$ for all $f \in C^\vee$ and $c \in C$.
\end{proof}

\begin{definition}{\rm (Ginzburg \cite{Gin06} and Van den Bergh \cite{VdB15})
A homologically smooth dg algebra $A$ is said to be {\it Calabi-Yau of dimension $d$} ({\it $d$-CY} for short)
if there exists an isomorphism $A \cong A^![d]$ in $\mathcal{D}A^e$
where $A^{!} := \RHom_{A^e}(A, A^e)$ is the dualizing dg $A$-bimodule.
}\end{definition}

Thanks to Proposition \ref{Proposition-Semi-proj-Res-Omega},
the semi-projective resolution $\Omega \otimes_{\pi} B \Omega A^\vee \otimes_{\pi} \Omega$ of the dg $\Omega$-bimodule $\Omega$
can be replaced with $\Omega \otimes_{\iota} A^\vee \otimes_{\iota} \Omega$.

\begin{theorem} \label{Theorem-Sym-Dual-CY}
Let $A$ be a finite dimensional complete typical derived $d$-symmetric dg $K$-ring.
Then $\Omega A^\vee$ is a $d$-CY dg algebra.
\end{theorem}

\begin{proof} Since $A$ is a finite dimensional complete dg $K$-ring,
by Theorem \ref{Theorem-Homologically-Smooth}, $\Omega A^\vee$ is homologically smooth.
The derived $d$-symmetry of $A$ implies that there is
an isomorphism $A \cong A^\vee[-d]$ in $\mathcal{D}A^e$ which gives a dg $A$-bimodule quasi-isomorphism
$\gamma : A \otimes_\pi BA \otimes_\pi A \rightarrow A^\vee[-d]$.
Since $A$ is finite dimensional and typical, there is a dg $A^\vee$-bicomodule isomorphism
$(A\otimes_\pi BA \otimes_\pi A)^\vee \cong A^\vee\otimes_{\iota} \Omega \otimes_{\iota} A^\vee$.
Taking the graded duals of dg $A$-bimodule quasi-isomorphisms $A \xleftarrow{\tilde{\mu}} A\otimes_\pi BA \otimes_\pi A \xrightarrow{\gamma} A^\vee[-d]$,
we get dg $A^\vee$-bicomodule quasi-isomorphisms $A^\vee\xrightarrow{\tilde{\mu}^\vee} A^\vee\otimes_{\iota} \Omega \otimes_{\iota} A^\vee \xleftarrow{\gamma^\vee}(A^\vee[-d])^\vee= A[d].$
Now we show that
$$\Omega\otimes_\iota A^\vee \otimes_{\iota}\Omega\xrightarrow{\id\otimes \tilde{\mu}^\vee \otimes \id}
\Omega \otimes_{\iota}(A^\vee\otimes_{\iota} \Omega \otimes_{\iota} A^\vee)\otimes_{\iota}\Omega \xleftarrow{\id\otimes \gamma^\vee\otimes \id} \Omega\otimes_{\iota}A[d]\otimes_{\iota}\Omega$$
are dg $\Omega$-bimodule quasi-isomorphisms.

Since $\tilde{\mu}^\vee : A^\vee \rightarrow A^\vee\otimes_{\iota} \Omega \otimes_{\iota} A^\vee=:\tilde{\Omega}$
is a dg $A^\vee$-bicomodule quasi-isomorphism,
$\id\otimes \tilde{\mu}^\vee \otimes \id : \Omega\otimes_{\iota}A^\vee \otimes_{\iota} \Omega \rightarrow \Omega\otimes_{\iota} \tilde{\Omega} \otimes_{\iota}\Omega$
is a dg $\Omega$-bimodule morphism. Now we show that it is a dg $\Omega$-bimodule quasi-isomorphism.
The dg $\Omega$-bimodule $\Omega\otimes_{\iota}A^\vee\otimes_{\iota}\Omega$ admits a dg $\Omega$-bimodule filtration
$$\dots \subset F_{-p-1}\subset F_{-p}\subset \dots \subset F_{-1}\subset F_0= \Omega\otimes_{\iota}A^\vee\otimes_{\iota}\Omega$$
where $F_{-p} := \bigoplus\limits_{m+n\geq p}(s^{-1}\overline{A}^\vee)^{\otimes m} \otimes A^\vee \otimes (s^{-1}\overline{A}^\vee)^{\otimes n}$
for all $p \in \mathbb{N}$,
which determines a spectral sequence:
$$E_{pq}^0 :=(F_{-p})_{p+q}/(F_{-p-1})_{p+q}
= (\bigoplus\limits_{m+n=p}(s^{-1}\overline{A}^\vee)^{\otimes m} \otimes A^\vee \otimes (s^{-1}\overline{A}^\vee)^{\otimes n})_{p+q}.$$
Recall that the differential of $\Omega$ is $d_0+d_1$ where $d_0$ preserves the weight of an element in $\Omega$
and $d_1$ maps ${(s^{-1}\overline{A}^\vee)}^{\otimes n}$ to $(s^{-1}\overline{A}^\vee)^{\otimes n+1}$.
The differential of $\Omega\otimes_{\iota}A^\vee\otimes_{\iota}\Omega$ is
$(d_0+d_1)\otimes \id\otimes \id +\id \otimes d_{A^\vee}\otimes \id+\id\otimes \id\otimes (d_0+d_1)+\id\otimes d_{\iota}^r-d_{\iota}^l\otimes \id$,
so the induced differential on $E_{p\bullet}^0$ is just $d_0\otimes \id\otimes \id +\id\otimes d_{A^\vee}\otimes \id+\id\otimes \id\otimes d_0$.
Since $A$ is finite dimensional and typical, $s^{-1}\overline{A}^\vee$ is either non-positive or non-negative.
Then the filtration is bounded,
and thus the spectral sequence is bounded and converges to the homology of $\Omega\otimes_{\iota}A^\vee \otimes_{\iota} \Omega$.
Similarly, the dg $\Omega$-bimodule $\Omega \otimes_{\iota} \tilde{\Omega} \otimes_{\iota}\Omega$ admits a bounded filtration $\{F_{-p}^\prime \ | \ p \in \mathbb{N}\}$
which determines a bounded spectral sequence $E_{pq}^{\prime}$ converging to the homology of $\Omega \otimes_{\iota} \tilde{\Omega} \otimes_{\iota}\Omega$.
The dg $\Omega$-bimodule morphism $\id\otimes \tilde{\mu}^\vee \otimes \id : \Omega \otimes_{\iota}A^\vee \otimes_{\iota} \Omega
\rightarrow \Omega \otimes_{\iota}\tilde{\Omega} \otimes_{\iota} \Omega$ induces a morphism of spectral sequences
$E_{pq}^0(\id\otimes \tilde{\mu}^\vee \otimes \id) : E_{pq}^0 \rightarrow E_{pq}^{\prime 0}$.
Since the differential on $E_{p\bullet}^0$ (resp. $E_{p\bullet}^{\prime 0}$) is
$d_0 \otimes \id \otimes \id +\id \otimes d_{A^\vee} \otimes \id+\id\otimes \id\otimes d_0$
(resp. $d_0 \otimes \id \otimes \id+\id \otimes d_{\tilde{\Omega}} \otimes \id+\id\otimes \id\otimes d_0$),
by K\"{u}nneth formula (\cite[Theorem 3.6.3]{Wei94}), the induced morphism on the first page
$E_{pq}^1(\id\otimes \tilde{\mu}^\vee \otimes \id) : E_{pq}^1 = H_q(E^0_{p\bullet}) \rightarrow E_{pq}^{\prime 1} = H_q(E^{\prime0}_{p\bullet})$ is an isomorphism.
By Comparison Theorem (\cite[Theorem 5.2.12]{Wei94}), $\id\otimes \tilde{\mu}^\vee \otimes \id$ induces an isomorphism on the homologies of
$\Omega \otimes_{\iota} A^\vee \otimes_{\iota} \Omega$ and $\Omega \otimes_{\iota} \tilde{\Omega} \otimes_{\iota} \Omega$.
Hence, $\id\otimes \tilde{\mu}^\vee\otimes \id$ is a dg $\Omega$-bimodule quasi-isomorphism, and thus an isomorphism in $\mathcal{D}\Omega^e$.
Similarly, $\id\otimes \gamma^\vee \otimes \id$ is also an isomorphism in $\mathcal{D}\Omega^e$.
Thus $\Omega\otimes_{\iota} A[d] \otimes_{\iota}\Omega \cong \Omega\otimes_{\iota} A^\vee \otimes_{\iota}\Omega$ in $\mathcal{D}\Omega^e$.

By Proposition \ref{Proposition-Semi-proj-Res-Omega} and Lemma \ref{Lemma-BimodIso}, we have
$$\begin{array}{ll} \RHom_{\Omega^e}(\Omega,\Omega^e)[d] & = \Hom_{\Omega^e}(\Omega\otimes_{\iota} A^\vee \otimes_{\iota}\Omega ,\Omega^e)[d] \\
& \cong (\Omega\otimes_{\iota} A \otimes_{\iota}\Omega)[d] \\
& \cong \Omega\otimes_{\iota} A[d] \otimes_{\iota}\Omega \\
& \cong \Omega\otimes_{\iota} A^\vee \otimes_{\iota}\Omega \\
& \cong \Omega \end{array}$$
in $\mathcal{D}\Omega^e$. Therefore, $\Omega$ is a $d$-Calabi-Yau dg algebra.
\end{proof}

\begin{remark}{\rm
If $A$ is a finite dimensional complete strictly $d$-symmetric dg $K$-ring then $\Omega A^\vee$ is a $d$-CY dg algebra.
Note that we need not the assumption that $A$ is typical here.
Indeed, since $A[d]$ and $A^\vee$ are isomorphic as dg $A$-bimodules, they are isomorphic as $A^\vee$-bicomodules.
Thus $\Omega \otimes_{\iota} A[d] \otimes_{\iota}\Omega$ and $\Omega \otimes_{\iota}A^\vee \otimes_{\iota} \Omega$
are isomorphic as dg $\Omega$-bimodules. Then we can use the last paragraph of the proof of Theorem \ref{Theorem-Sym-Dual-CY}.
}\end{remark}

\subsection{Batalin-Vilkovisky algebra isomorphisms}

In this part, we show that for a finite dimensional complete typical $d$-symmetric dg $K$-ring $A$,
the Hochschild cohomology $HH^\bullet(\Omega A^\vee)$ of $\Omega A^\vee$ is a BV algebra,
and $HH^\bullet(\Omega A^\vee)$ and $HH^\bullet(A)$ are isomorphic as BV algebras.

\bigskip

\noindent{\bf Batalin-Vilkovisky algebras.} Batalin-Vilkovisky algebra structure is subtler than Gerstenhaber algebra structure.

\begin{definition}{\rm
A {\it Batalin-Vilkovisky (BV for short) algebra} $(A^\bullet, \cdot, [-,-], \Delta)$ is a Gerstenhaber algebra $(A^\bullet, \cdot, [-,-])$
equipped with an operator $\Delta : A^\bullet \rightarrow A^\bullet$ of degree $-1$
such that $\Delta^2=0$ and $[-,-]$ is the deviation from being a derivation for the product $\cdot$, i.e.,
$[a,b] = (-1)^{|a|} \Delta(ab) - (-1)^{|a|} \Delta(a)b - a\Delta(b)$ for all $a,b \in A^\bullet$.
}\end{definition}

Let $(A^\bullet, \cdot, [-,-], \Delta)$ be a BV algebra.
It follows from $\Delta^2=0$ that $\Delta$ is a derivation on the Gerstenhaber bracket $[-,-]$, i.e.,
$\Delta[a,b] = [\Delta(a),b]+(-1)^{|a|+1}[a,\Delta(b)]$ for all $a,b \in A^\bullet$.
Moreover, we have the following {\it 7-term relation} \cite{Get94}:
\begin{eqnarray*}
\Delta(abc) = & \Delta(ab)c+(-1)^{|a|}a\Delta(bc)+(-1)^{(|a|+1)|b|}b\Delta(ac)\\
& \quad -\Delta(a)bc-(-1)^{|a|}a\Delta(b)c-(-1)^{|a|+|b|}ab\Delta(c)
\end{eqnarray*}
for all $a,b,c \in A^\bullet$.

\medskip

The 7-term relation provides an equivalent definition of BV algebra:

\begin{definition} {\rm
A {\it BV algebra} $(A^\bullet, \cdot, \Delta)$ is a graded commutative associative algebra $(A^\bullet, \cdot)$
endowed with an operator $\Delta: A^\bullet \rightarrow A^\bullet$ of degree $-1$ satisfying $\Delta^2=0$ and the 7-term relation.
}\end{definition}

Indeed, if $[a,b] := (-1)^{|a|}\Delta(ab)-(-1)^{|a|}\Delta(a)b - a\Delta(b)$ then the graded Leibniz rule for $[-,-]$ is equivalent to the 7-term relation.
Furthermore, it is easy to show that $[-,-]$ is a Gerstenharber bracket by applying the 7-term relation.

\begin{definition}{\rm
A {\it BV algebra morphism} from a BV algebra $(A^\bullet, \cdot, [-,-], \Delta)$ to a BV algebra $(A'^\bullet, \cdot', [-,-]', \Delta')$
is a graded algebra morphism $f : A^\bullet \rightarrow A'^\bullet$ satisfying $f \Delta = \Delta' f$.
}\end{definition}

\bigskip

\noindent{\bf BV algebras for $d$-CY dg algebras.} In order to define the BV operator (differential) $\Delta$
on the Hochschild cohomology of a $d$-CY dg algebra, we need Van den Bergh dual whose proof is also included here for later use.

\begin{theorem} \label{Theorem-VdB-Dual} {\rm (Van den Bergh \cite{VdB98})}
Let $A$ be a $d$-CY dg algebra. Then there is a graded $k$-vector space isomorphism:
$$D : HH_{d-\bullet}(A) \rightarrow HH^{\bullet}(A).$$
\end{theorem}

\begin{proof} Since $A$ is $d$-CY, there is an isomorphism $A \cong A^![d]$ in $\mathcal{D}A^e$.
Then the following composition $D$ is as required:
\begin{align*}
HH_{d-\bullet}(A) & = H^{\bullet-d}(A \otimes^L_{A^e} A) = H^\bullet(A \otimes^L_{A^e} A[-d]) \\
& \cong H^\bullet(A \otimes^L_{A^e} A^!) = H^\bullet(A \otimes^L_{A^e}\RHom_{A^e}(A,A^e)) \\
& \cong H^\bullet(\RHom_{A^e}(A,A)) = HH^{\bullet}(A).
\end{align*}
\qedhere
\end{proof}

Indeed, Van den Bergh dual $D$ is a graded $HH^\bullet(A)$-module isomorphism.

\begin{lemma} \label{Lemma-VdBDualModMor} {\rm (Ginzburg \cite[Theorem 3.4.3 (i)]{Gin06} and Abbaspour \cite[Lemma 3.7]{Abb15})}
Let $A$ be a $d$-CY dg algebra.
Then $D(i_fa)=f \cup Da$ for all $f \in HH^\bullet(A)$ and $a \in HH_{d-\bullet}(A)$,
i.e., the following diagram is commutative
$$\xymatrix{ HH^\bullet(A) \otimes_k HH_{d-\bullet}(A) \ar[r]^-{\cap} \ar[d]^-{\id \otimes D}_-{\cong} & HH_{d-\bullet}(A) \ar[d]^-{D}_-{\cong} \\
HH^\bullet(A) \otimes_k HH^\bullet(A) \ar[r]^-{\cup} & HH^\bullet(A) } \ \ \
\xymatrix{ f \otimes a \ar@{|->}[r] \ar@{|->}[d] & i_fa \ar@{|->}[d] \\
f \otimes Da \ar@{|->}[r] & f \cup Da = D(i_fa) }$$
which implies that $D : HH_{d-\bullet}(A) \rightarrow HH^\bullet(A)$ is a graded $HH^\bullet(A)$-module isomorphism.
\end{lemma}

\begin{proof}
Let $P = A \otimes_\pi BA \otimes_\pi A$ be the two-sided bar resolution of $A$ and
$$\tilde{\mu} : P=A \otimes_\pi BA \otimes_\pi A \xrightarrow{\id \otimes \varepsilon \otimes \id} A \otimes K \otimes A = A \otimes A \xrightarrow{\mu} A.$$
Then the isomorphism $A \cong A^![d]$ in $\mathcal{D}A^e$ provides
a dg $A$-bimodule quasi-isomorphism $\sigma : P \rightarrow \Hom_{A^e}(P,A^e)[d]$.

Firstly, the evaluation map
$$\ev : \End_{A^e}(P) \otimes_k (P \otimes_{A^e}P) \rightarrow P \otimes_{A^e}P, \; f \otimes(p \otimes p') \mapsto f(p) \otimes p',$$
and the composition map
$$\comp : \End_{A^e}(P) \otimes_k \End_{A^e}(P) \rightarrow \End_{A^e}(P), \ f \otimes g \mapsto f \circ g,$$
satisfy the following commutative diagram:
$$\xymatrix{ \End_{A^e}(P) \otimes_k (P \otimes_{A^e} P)[-d] \ar[r]^-{\ev[-d]} \ar[d]^-{\cong}
& (P \otimes_{A^e} P)[-d] \ar[d]^-{\cong} \\
\End_{A^e}(P) \otimes_k (P \otimes_{A^e} P[-d]) \ar[d]^-{\id \otimes (\id \otimes \sigma[-d])}
& P \otimes_{A^e} P[-d] \ar[d]^-{\id \otimes \sigma[-d]} \\
\End_{A^e}(P) \otimes_k (P \otimes_{A^e} \Hom_{A^e}(P,A^e)) \ar[d]^-{\cong} & P \otimes_{A^e} \Hom_{A^e}(P,A^e) \ar[d]^-{\cong} \\
\End_{A^e}(P) \otimes_k \End_{A^e}(P) \ar[r]^-{\comp} & \End_{A^e}(P).
}$$
Note that the left column is just the tensor product of identity and the right column.

Secondly, the evaluation map $\ev : \End_{A^e}(P) \otimes_k (P \otimes_{A^e}P) \rightarrow P \otimes_{A^e}P$
induces the operator $i_f$ by the following commutative diagram:
$$\xymatrix{ \End_{A^e}(P) \otimes (P \otimes_{A^e} P) \ar[r]^-{\ev} \ar[d]^-{\tilde{\mu}_* \otimes \id} & P \otimes_{A^e} P \ar[d]^-{\tilde{\mu} \otimes \id} \\
\Hom_{A^e}(P,A) \otimes_k (P \otimes_{A^e} P) \ar[r]^-{\ev '} & A \otimes_{A^e} P \\
\Hom_{A^e}(P,A) \otimes_k (A \otimes_{A^e} (P \otimes_A P)) \ar[u]_-{\cong} & \\
\Hom_{A^e}(P,A) \otimes_k (A \otimes_{A^e} P) \ar[u]_-{\id \otimes(\id \otimes \Delta)} & \\
\Hom_{K^e}^\pi(BA,A)^\op \otimes_k (A \otimes_{K^e}^\pi BA) \ar[r]^-{\cap} \ar[u]^-{\cong} & A \otimes_{K^e}^\pi BA \ar[uuu]^-{\cong} }$$
where the map
$\ev ' : \Hom_{A^e}(P,A) \otimes_k (P \otimes_{A^e} P) \rightarrow A \otimes_{A^e} P, \ f \otimes (p \otimes p') \mapsto f(p) \otimes p',$
the map $\Delta : P \rightarrow P \otimes_A P, \; 1[a_1|\cdots|a_n]1 \mapsto \sum\limits_{i=0}^n(1[a_1|\cdots|a_i]1)\otimes(1[a_{i+1}|\cdots|a_n]1),$
and the map
$\cap : \Hom_{K^e}^\pi(BA,A) \otimes_k (A \otimes_{K^e}^\pi BA) \rightarrow A \otimes_{K^e}^\pi BA, \ f \otimes a \mapsto i_f(a).$
Note that the map $\tilde{\mu} \otimes \id : P \otimes_{A^e} P \rightarrow A \otimes_{A^e} P$
and the map $A \otimes_{A^e} P \xrightarrow{\id \otimes \Delta} A \otimes_{A^e} (P \otimes_A P) \xrightarrow{\cong} P \otimes_{A^e} P$
are inverse to each other when taking homologies.

Thirdly, the composition map $\comp : \End_{A^e}(P) \otimes_k \End_{A^e}(P) \rightarrow \End_{A^e}(P)$
induces the opposite cup product on $\Hom_{K^e}^\pi(BA,A)$ by the following commutative diagram:
$$\xymatrix{ \Hom_{K^e}^\pi(BA,A) \otimes_k \Hom_{K^e}^\pi(BA,A) \ar[r]^-{\cup^\op} \ar[d]_-{\cong} & \Hom_{K^e}^\pi(BA,A) \\
\Hom_{A^e}(P,A) \otimes_k \Hom_{A^e}(P,A) \ar[d]^-{\varsigma \otimes \varsigma}_-{\cong} & \Hom_{A^e}(P,A) \ar[u]_-{\cong} \\
\End_{A^e}(P) \otimes_k \End_{A^e}(P) \ar[r]^-{\comp} & \End_{A^e}(P). \ar[u]^-{\tilde{\mu}_*}_-{\cong}
}$$
where the map
$\cup^\op : \Hom_{K^e}^\pi(BA,A) \otimes_k \Hom_{K^e}^\pi(BA,A) \rightarrow \Hom_{K^e}^\pi(BA,A),\ f\otimes g \mapsto (-1)^{|f||g|}g\cup f$,
and the map $\varsigma : \Hom_{A^e}(P,A) \stackrel{- \otimes_A P}{\longrightarrow} \Hom_{A^e}(P \otimes_A P, A \otimes_A P)
\stackrel{\Delta^*}{\rightarrow} \Hom_{A^e}(P, A \otimes_A P) \stackrel{\mu^l_*}{\rightarrow} \Hom_{A^e}(P,P)$ is given by
$\varsigma(f) := \mu^l\circ( f \otimes \id)\circ\Delta : P \stackrel{\Delta}{\rightarrow} P \otimes_A P
\stackrel{ f\otimes\id}{\rightarrow} A \otimes_A P \stackrel{\mu^l}{\rightarrow} P$ for all $f \in \Hom_{A^e}(P,A)$.
Note that the map $\tilde{\mu}_* : \End_{A^e}(P) \rightarrow \Hom_{A^e}(P,A)$ and the map $\varsigma : \Hom_{A^e}(P,A) \rightarrow \End_{A^e}(P)$
are inverse to each other when taking homologies due to $\tilde{\mu}_* \circ \varsigma=1$.

Finally, using three commutative diagrams above, by taking cohomologies,
we can obtain the desired commutative diagram in the theorem.
\end{proof}

The Connes operator $B$ on $HH_{\bullet}(A)$ induces an operator of degree $-1$
$$\Delta := -D \circ B \circ D^{-1}$$
on $HH^{\bullet}(A)$ by the following commutative diagram:
$$\xymatrix{ HH_{d-\bullet}(A) \ar[r]^-{-B} \ar[d]_-{D} & HH_{d-\bullet}(A) \ar[d]^-D \\
HH^{\bullet}(A)\ar@{.>}[r]^-{\Delta} & HH^{\bullet}(A). }$$

\begin{lemma}\label{Lemma-CY[f,g]cupDa} {\rm (Ginzburg \cite[Theorem 3.4.3 (ii)]{Gin06} and Abbaspour \cite[Corollary 3.9]{Abb15})}
Let $A$ be a d-CY dg algebra. Then
$$\begin{array}{ll} [f,g] \cup Da = & (-1)^{|f|} \Delta(f\cup g\cup Da) - f\cup\Delta(g\cup Da) \\
& \quad + (-1)^{(|f|+1)(|g|+1)}g\cup\Delta(f\cup Da) + (-1)^{|g|}f\cup g\cup\Delta Da \end{array}$$
for all $f,g \in HH^\bullet(A)$ and $a \in HH_{d-\bullet}(A)$.
\end{lemma}

\begin{proof}
Applying Lemma \ref{Lemma-i_f-L_f}, we have
$$\begin{array}{ll}
&i_{[f,g]} \\
=& [i_f, L_g]\\
=& i_f L_g - (-1)^{|f|(|g|+1)} L_g i_f\\
=& i_f [ B ,i_g] - (-1)^{|f|(|g|+1)} [ B ,i_g] i_f\\
=& i_f B i_g - (-1)^{|g|}i_f i_g B
- (-1)^{|f|(|g|+1)} B i_g i_f + (-1)^{|f|(|g|+1)+|g|}i_g B i_f\\
=& (-1)^{|f|+1} B i_{f\cup g} + i_f B i_g
- (-1)^{(|f|+1)(|g|+1)}i_g B i_f - (-1)^{|g|}i_{f \cup g} B.
\end{array}$$

By Lemma \ref{Lemma-VdBDualModMor}, we further have
$$\begin{array}{ll}
& [f,g] \cup Da \\
=& D(i_{[f,g]}a)\\
=& (-1)^{|f|+1}D B i_{f\cup g}a + Di_f B i_ga \\
& \quad - (-1)^{(|f|+1)(|g|+1)}Di_g B i_fa - (-1)^{|g|}Di_{f \cup g} B a\\
=& (-1)^{|f|}\Delta D(i_{f\cup g}a) + f\cup D B i_ga \\
& \quad - (-1)^{(|f|+1)(|g|+1)}g\cup D B i_fa - (-1)^{|g|}f\cup g\cup D B a\\
=& (-1)^{|f|}\Delta (f\cup g\cup Da) - f\cup \Delta D( i_ga) \\
& \quad + (-1)^{(|f|+1)(|g|+1)}g\cup\Delta D(i_fa) + (-1)^{|g|}f\cup g\cup \Delta Da\\
=& (-1)^{|f|}\Delta(f\cup g\cup Da) - f\cup\Delta(g\cup Da) \\
& \quad + (-1)^{(|f|+1)(|g|+1)}g\cup\Delta(f\cup Da) + (-1)^{|g|}f\cup g\cup \Delta Da.
\end{array}$$
\end{proof}

\begin{theorem}\label{Theorem-CYBV} {\rm (Ginzburg \cite[Theorem 3.4.3 (ii)]{Gin06} and Abbaspour \cite[Theorem 3.10]{Abb15})}
Let $A$ be a $d$-CY dg algebra satisfy $\Delta(1_{HH^{\bullet}(A)}) = 0$.
Then $(HH^{\bullet}(A), \cup, \Delta)$ is a BV algebra, i.e.,
$$[f,g] = (-1)^{|f|}\Delta(f\cup g) - (-1)^{|f|}\Delta(f)\cup g - f\cup\Delta(g)$$
for all $f,g \in HH^\bullet(A)$.
\end{theorem}

\begin{proof}
Since $D$ is an isomorphism, there exists $a \in HH_d(A)$ such that
$Da = 1 \in HH^0(A)$. Then the theorem  follows from Lemma \ref{Lemma-CY[f,g]cupDa} and the condition $\Delta(1_{HH^{\bullet}(A)})=0.$
\end{proof}

\bigskip

\noindent{\bf BV algebras for (derived) $d$-symmetric dg $K$-rings.}
Let $A$ be a derived $d$-symmetric dg $K$-ring.
It is known that the dg $A$-bimodules morphism $\tilde{\mu}:A\otimes_{\pi}BA\otimes_{\pi}A\xrightarrow{\id
\otimes\varepsilon\otimes \id}A\otimes K \otimes A = A\otimes A\xrightarrow{\mu}A$ is a quasi-isomorphism.
Furthermore, the isomorphism $A\cong A^\vee[-d]$ in $\mathcal{D}A^e$ gives a quasi-isomorphism of dg $A$-bimodules
$\gamma: A\otimes_\pi BA\otimes_{\pi} A\rightarrow A^\vee[-d]$.
Thus we have the following series of quasi-isomorphisms:
$$\begin{array}{lll}
(A\otimes_{K^e}^{\pi} BA)^\vee[-d]&\cong &(A\otimes_{A^e}(A\otimes_\pi BA\otimes_\pi A))^\vee[-d] \\
& \cong & \Hom_{A^e}(A \otimes_\pi BA \otimes_\pi A,A^\vee[-d]) \\
& \xleftarrow{\gamma_*} & \Hom_{A^e}(A\otimes_\pi BA\otimes_\pi A,A\otimes_\pi BA\otimes_\pi A) \\
& \xrightarrow{\tilde{\mu}_*} & \Hom_{A^e}(A\otimes_\pi BA\otimes_\pi A,A)\\
& \cong & \Hom_{K^e}^\pi(BA,A).
\end{array}$$
By taking cohomologies, we obtain a graded $k$-vector space isomorphism
$$D: HH_{\bullet-d}(A)^\vee \rightarrow HH^{\bullet}(A).$$

Indeed, $D$ is a graded $HH^{\bullet}(A)$-module isomorphism.

\begin{lemma} \label{Lemma-DualModMor} {\rm (Abbaspour \cite[Lemma 4.2]{Abb15})}
Let $A$ be a derived $d$-symmetric dg $K$-ring.
Then $D(i_f^\vee \phi) = f \cup D\phi$ for all $f \in HH^\bullet(A)$ and $\phi\in HH_{\bullet-d}(A)^\vee$, i.e.,
the following diagram is commutative:
$$\xymatrix{ HH^\bullet(A)\otimes_k HH_{\bullet-d}(A)^\vee \ar[r]^-{} \ar[d]^-{\id \otimes D} & HH_{\bullet-d}(A)^\vee \ar[d]^-{D}\\
HH^\bullet(A) \otimes_k HH^\bullet(A) \ar[r]^-{\cup} & HH^\bullet(A) }
\xymatrix{ f \otimes \phi \ar@{|->}[r] \ar@{|->}[d] & i_f^\vee \phi \ar@{|->}[d] \\
f \otimes D\phi \ar@{|->}[r] & f \cup D\phi = D(i_f^\vee \phi) }$$
which implies that the map $D : HH_{\bullet-d}(A)^\vee \rightarrow HH^\bullet(A)$ is a graded $HH^\bullet(A)$-module isomorphism.
\end{lemma}

\begin{proof}
Let $P = A \otimes_\pi BA \otimes_\pi A$ be the two-sided bar resolution of $A$ and
$\tilde{\mu} : P=A \otimes_\pi BA \otimes_\pi A \xrightarrow{\id \otimes \varepsilon \otimes \id} A \otimes K \otimes A = A \otimes A \xrightarrow{\mu} A$.
Then the dg $A$-bimodule morphism $A^\vee[-d]\xrightarrow{\tilde{\mu}^\vee[-d]} P^\vee[-d]$ is a quasi-isomorphism.
The isomorphism $A \cong A^\vee[-d]$ in $\mathcal{D}A^e$ provides a dg $A$-bimodule quasi-isomorphism $\gamma : P \rightarrow A^\vee[-d]$.
Let $\kappa$ be the composition $ P\xrightarrow{\gamma} A^\vee[-d]\xrightarrow{\tilde{\mu}^\vee[-d]} P^\vee[-d]$.
Then $\kappa$ is a quasi-isomorphism of dg $A$-bimodules.

Firstly, the map $\alpha : (P \otimes_{A^e}P)^\vee \otimes_k \End_{A^e}(P) \rightarrow (P \otimes_{A^e}P)^\vee$
given by $\alpha(\phi \otimes f)(p \otimes p') := \phi(f(p) \otimes p')$,
and the map $\beta : \Hom_{A^e}(P, P^\vee[-d]) \otimes_k \End_{A^e}(P) \rightarrow \Hom_{A^e}(P, P^\vee[-d])$
given by $\beta(\psi \otimes f)(p) := \psi(f(p))$,
satisfy the following commutative diagram:
$$\xymatrix{ (P \otimes_{A^e} P)^\vee[-d] \otimes_k \End_{A^e}(P) \ar[r]^-{\alpha[d]} \ar[d]^-{\cong} & (P \otimes_{A^e} P)^\vee[-d] \ar[d]^-{\cong} \\
\Hom_{A^e}(P, P^\vee[-d]) \otimes_k \End_{A^e}(P) \ar[r]^-{\beta} & \Hom_{A^e}(P, P^\vee[-d]) \\
\End_{A^e}(P) \otimes_k \End_{A^e}(P) \ar[r]^-{\comp} \ar[u]_-{\kappa_* \otimes \id} & \End_{A^e}(P). \ar[u]_-{\kappa_*}
}$$

Secondly, the map $\alpha : (P \otimes_{A^e}P)^\vee \otimes_k \End_{A^e}(P) \rightarrow (P \otimes_{A^e}P)^\vee$
induces the map $i_f^\vee$ by the following commutative diagram:
{\footnotesize $$\xymatrix{ (P \otimes_{A^e} P)^\vee \otimes_k \End_{A^e}(P) \ar[r]^-{\alpha} & (P \otimes_{A^e} P)^\vee \\
(A \otimes_{A^e} P)^\vee \otimes_k \End_{A^e}(P) \ar[u]_-{(\tilde{\mu} \otimes \id)^\vee \otimes \id} \ar[d]^-{\id \otimes\tilde{\mu}_*} & \\
(A \otimes_{A^e} P)^\vee \otimes_k \Hom_{A^e}(P,A) \ar[r]^-{\alpha'} \ar[ddd]^-{\cong} & (P \otimes_{A^e} P)^\vee \ar@{=}[uu] \ar[d]^-{\cong}\\
& (A \otimes_{A^e} (P \otimes_A P))^\vee \ar[d]^-{(\id \otimes \Delta)^\vee} \\
& (A \otimes_{A^e} P)^\vee \ar[d]^-{\cong} \\
(A \otimes_{K^e}^\pi BA)^\vee \otimes_k \Hom_{K^e}^\pi(BA,A)^{\op} \ar[r]^-{\cap'} & (A \otimes_{K^e}^\pi BA)^\vee }$$}
where the map $\alpha' : (A \otimes_{A^e} P)^\vee \otimes_k \Hom_{A^e}(P,A) \rightarrow (P \otimes_{A^e} P)^\vee$
is given by $\alpha'(\phi \otimes f)(p \otimes p') := \phi(fp \otimes p')$,
the map $\Delta : P \rightarrow P \otimes_A P, \; 1[a_1|\cdots|a_n]1 \mapsto \sum\limits_{i=0}^n(1[a_1|\cdots|a_i]1)\otimes(1[a_{i+1}|\cdots|a_n]1)$,
and the map $\cap' : (A \otimes_{K^e}^\pi BA)^\vee \otimes_k \linebreak \Hom_{K^e}^\pi(BA,A)^{\op}
\rightarrow (A \otimes_{K^e}^\pi BA)^\vee, \ \phi \otimes f\mapsto \phi i_f$.
In fact, the commutative diagram above is dual to the second one in the proof of Lemma \ref{Lemma-VdBDualModMor}.

Thirdly, the composition map $\comp : \End_{A^e}(P) \otimes_k \End_{A^e}(P) \rightarrow \End_{A^e}(P)$
induces the opposite cup product on $\Hom_{K^e}^\pi(BA,A)$ which has been mentioned in Lemma \ref{Lemma-VdBDualModMor}.

Finally, using two commutative diagrams above and the third commutative diagram in the proof of Lemma \ref{Lemma-VdBDualModMor}, by taking cohomologies,
we can obtain the desired commutative diagram in the theorem.
\end{proof}

The Connes operator $B$ on $HH_\bullet(A)$ induces an operator of degree $-1$
$$\Delta := -D \circ B ^\vee\circ D^{-1}$$
on $HH^{\bullet}(A)$ by the following commutative diagram:
$$\xymatrix{ HH_{\bullet-d}(A)^\vee \ar[r]^-{-B^\vee} \ar[d]_{D} & HH_{\bullet-d}(A)^\vee \ar[d]^-D\\
HH^{\bullet}(A) \ar@{.>}[r]^-{\Delta} & HH^{\bullet}(A). }$$

\begin{lemma} \label{Lemma-Sym[f,g]cupDphi} {\rm (Abbaspour \cite[Corollary 4.3]{Abb15})}
Let $A$ be a derived $d$-symmetric dg $K$-ring. Then
$$\begin{array}{ll} [f,g] \cup D\phi = & (-1)^{|f|}\Delta(f \cup g \cup D \phi) - f \cup \Delta(g \cup D \phi) \\
& \quad + (-1)^{(|f|+1)(|g|+1)}g\cup\Delta(f\cup D\phi) + (-1)^{|g|}f\cup g\cup \Delta D\phi \end{array}$$
for all $f,g \in HH^\bullet(A)$ and $\phi \in HH_{\bullet-d}(A)^\vee$.
\end{lemma}

\begin{proof}
By the proof of Lemma \ref{Lemma-CY[f,g]cupDa}, we obtain
$$i_{[f,g]} = (-1)^{|f|+1} B i_{f\cup g} + i_f B i_g - (-1)^{(|f|+1)(|g|+1)}i_g B i_f - (-1)^{|g|}i_{f \cup g} B$$
for all $f,g \in HH^\bullet(A)$. Thus
$$i_{[f,g]}^\vee = (-1)^{|f|+1}B^\vee i_{f \cup g}^\vee + i_f^\vee B^\vee i_g^\vee -
(-1)^{(|f|+1)(|g|+1)}i_g^\vee B^\vee i_f^\vee - (-1)^{|g|} i_{f \cup g}^\vee B^\vee$$
for all $f,g \in HH^\bullet(A)$. It follows from Lemma \ref{Lemma-DualModMor} that
$$\begin{array}{ll}
& [f,g] \cup D\phi \\
=& D(i_{[f,g]}^\vee\phi)\\
=& (-1)^{|f|+1}D B^\vee i_{f\cup g}^\vee\phi + Di_f^\vee B^\vee i_g^\vee\phi \\
& \quad - (-1)^{(|f|+1)(|g|+1)}Di_g^\vee B^\vee i_f^\vee\phi - (-1)^{|g|}Di_{f \cup g}^\vee B^\vee \phi \\
=& (-1)^{|f|}\Delta D(i_{f\cup g}^\vee\phi) + f\cup D B^\vee i_g^\vee\phi \\
& \quad - (-1)^{(|f|+1)(|g|+1)}g\cup D B^\vee i_f^\vee\phi - (-1)^{|g|}f\cup g\cup D B^\vee \phi\\
=& (-1)^{|f|}\Delta (f\cup g\cup D\phi) - f\cup \Delta D( i_g^\vee\phi) \\
& \quad + (-1)^{(|f|+1)(|g|+1)}g\cup\Delta D(i_f^\vee\phi) + (-1)^{|g|}f\cup g\cup \Delta D\phi\\
=& (-1)^{|f|}\Delta(f\cup g\cup D\phi) - f\cup\Delta(g\cup D\phi) \\
& \quad + (-1)^{(|f|+1)(|g|+1)}g\cup\Delta(f\cup D\phi) + (-1)^{|g|}f\cup g\cup \Delta D\phi.
\end{array}$$
\qedhere
\end{proof}

For derived $d$-symmetric dg $K$-rings, we have the following general result:

\begin{theorem} \label{Theorem-DerSymBV}  {\rm(Abbaspour\cite[Theorem 4.6]{Abb15})}
Let $A$ be a derived $d$-symmetric dg $K$-ring satisfying $\Delta(1_{HH^\bullet(A)})=0$.
Then $(HH^{\bullet}(A), \cup, \Delta)$ is a BV algebra, i.e.,
$$[f,g] = (-1)^{|f|}\Delta(f\cup g) - (-1)^{|f|}\Delta(f)\cup g-f\cup \Delta(g)$$
for all $f,g \in HH^{\bullet}(A)$.
\end{theorem}

\begin{proof}
Let $\phi\in HH_{-d}(A)^\vee$ be the pre-image of $1\in HH^0(A)$ under the isomorphism $D$.
Then $D\phi=1$ and $\Delta D\phi=\Delta(1)=0$.
Thus the theorem follows from Lemma \ref{Lemma-Sym[f,g]cupDphi}.
\end{proof}

For $d$-symmetric dg $K$-rings, we have the following cleaner result:

\begin{theorem} \label{Theorem-SymBV} {\rm (Menichi \cite[Theorem 18]{Men09} and Abbaspour \cite[Theorem 4.7]{Abb15})}
Let $A$ be a $d$-symmetric dg $K$-ring.
Then $(HH^{\bullet}(A), \cup, \Delta)$ is a $BV$ algebra.
\end{theorem}

\begin{proof}
Let $\kappa : A \rightarrow A^\vee[-d]$ be a dg $A$-bimodule quasi-isomorphism.
Then the graded $k$-module isomorphism $D^{-1}=H^\bullet(\varrho)$ where $\varrho$ is the composition
$\Hom_{K^e}^\pi(BA,A) \xrightarrow{\cong} \Hom_{A^e}(A \otimes_\pi BA \otimes_\pi A, A)
\xrightarrow{\kappa_*} \Hom_{A^e}(A \otimes_\pi BA \otimes_\pi A, A^\vee[-d]) \linebreak \xrightarrow{\cong}
(A[d] \otimes_{A^e}(A \otimes_\pi BA \otimes_\pi A))^\vee \xrightarrow{\cong} (A[d] \otimes^\pi_{K^e}BA)^\vee.$
Let $\sigma = \varrho(\eta_A\varepsilon_{BA})$.
Then $\sigma(s^da \otimes b)=(-1)^{|a||b|}(s^d((\kappa\circ\eta_A\circ\varepsilon_{BA})(b)))(a) \in k$ for all $b\in BA$ and $a\in A$.
Since $\eta_A\varepsilon_{BA}$ is the unit of $\Hom_{K^e}^\pi(BA,A)$, it is a 0-cocycle.
Thus $\sigma$ is a 0-cocycle of $(A[d] \otimes^\pi_{K^e}BA)^\vee$ and $D[\sigma] = 1_{HH^\bullet(A)}$ when passing to cohomologies.
Note that the image of the Connes operator $B$, $\Im B \subseteq A\otimes_\pi \overline{BA}$.
Thus $B^\vee \sigma = \sigma B=0$, and further $B^\vee [\sigma] =0$ and $\Delta(1_{HH^\bullet(A)})=\Delta(D[\sigma])=-D(B^\vee [\sigma])=0$.
Then the theorem follows from Lemma \ref{Lemma-Sym[f,g]cupDphi} or Theorem \ref{Theorem-DerSymBV}.
\end{proof}

\begin{remark}{\rm
If $A$ is a finite dimensional derived $d$-symmetric dg $K$-ring,
we may replace the quasi-isomorphisms $\Hom_{A^e}(A \otimes_\pi BA \otimes_\pi A,A^\vee[-d])
\xleftarrow{\gamma_*} \Hom_{A^e}(A\otimes_\pi BA\otimes_\pi A,A\otimes_\pi BA\otimes_\pi A)
\xrightarrow{\tilde{\mu}_*} \Hom_{A^e}(A\otimes_\pi BA\otimes_\pi A,A)$
in the definition of the isomorphism $D: HH_{\bullet-d}(A)^\vee \rightarrow HH^{\bullet}(A)$
with the quasi-isomorphisms
$\Hom_{A^e}(A \otimes_\pi BA \otimes_\pi A,A^\vee[-d]) \xrightarrow{(\tilde{\mu}^\vee[-d])_*}
\Hom_{A^e}(A\otimes_\pi BA\otimes_\pi A,(A\otimes_\pi BA\otimes_\pi A)^\vee[-d])
\xleftarrow{(\gamma^\vee[-d])_*} \Hom_{A^e}(A\otimes_\pi BA\otimes_\pi A,A)$,
and correspondingly replace the dg $A$-bimodule quasi-isomorphism 
$\kappa=\tilde{\mu}^\vee[-d]\circ\gamma: P\xrightarrow{\gamma} A^\vee[-d]\xrightarrow{\tilde{\mu}^\vee[-d]} P^\vee[-d]$
in the proof of Lemma \ref{Lemma-DualModMor}
with its dual $\kappa=\gamma^\vee[-d]\circ\tilde{\mu}: P\xrightarrow{\tilde{\mu}} A=(A^\vee[-d])^\vee[-d]\xrightarrow{\gamma^\vee[-d]} P^\vee[-d]$.
After replacing, Lemma \ref{Lemma-DualModMor}, Lemma \ref{Lemma-Sym[f,g]cupDphi},
Theorem \ref{Theorem-DerSymBV} and Theorem \ref{Theorem-SymBV} still hold.
Henceforth, we will do these replacements so that the CY structure of $\Omega A^\vee$
considered in the following is just that given in Theorem \ref{Theorem-Sym-Dual-CY}.
}\end{remark}

For derived $d$-symmetric dg $K$-rings and their Koszul duals, we have the following general result:

\begin{theorem} \label{Theorem-DerSymCY-BVIso}
Let $A$ be a finite dimensional complete typical derived $d$-symmetric dg $K$-ring satisfying $\Delta(1_{HH^{\bullet}(A)})=0$. Then

{\rm (1)} $(HH^\bullet(\Omega A^\vee), \cup, \Delta)$ is a BV algebra;

{\rm (2)} $(HH^\bullet(\Omega A^\vee), \cup, \Delta)$ and $(HH^{\bullet}(A), \cup, \Delta)$ are isomorphic as BV algebras.
\end{theorem}

\begin{proof} It follows from Theorem \ref{Theorem-Sym-Dual-CY} that $\Omega A^\vee$ is a $d$-CY dg algebra.
By the assumption $\Delta(1_{HH^{\bullet}(A)})=0$ and Theorem \ref{Theorem-DerSymBV},
we know $(HH^{\bullet}(A), \cup, \Delta)$ is a BV algebra.
To show that $(HH^\bullet(\Omega A^\vee), \cup, \Delta)$ is a BV algebra, by Theorem \ref{Theorem-CYBV},
it is enough to prove $\Delta(1_{HH^{\bullet}(\Omega A^\vee)})=0$.
For this, we have to give the Van den Bergh dual $D$ and the BV operator $\Delta=-D \circ B \circ D^{-1}$ explicitly.
Let $P = A \otimes_\pi BA \otimes_\pi A$ be the two-sided bar resolution of $A$ and
$\tilde{\mu} : P=A \otimes_\pi BA \otimes_\pi A \xrightarrow{\id \otimes \varepsilon \otimes \id} A \otimes K \otimes A = A \otimes A \xrightarrow{\mu} A$.
The isomorphism $A \cong A^\vee[-d]$ in $\mathcal{D}A^e$ provides a dg $A$-bimodule quasi-isomorphism $\gamma : P \rightarrow A^\vee[-d]$.
By the assumption that $A$ is finite dimensional, $\gamma^\vee[-d] : A=(A^\vee[-d])^\vee[-d] \rightarrow P^\vee[-d]$
and $\tilde{\mu}^\vee[-d] : A^\vee[-d] \rightarrow P^\vee[-d]$ are dg $A^\vee$-bicomodule quasi-isomorphisms.
Keep in mind the proofs of Theorem \ref{Theorem-Sym-Dual-CY} and Theorem \ref{Theorem-VdB-Dual},
we can determine $D$, and thus $\Delta$, explicitly by the following commutative diagram:

{\footnotesize $$\xymatrix{
H^\bullet(\Omega \otimes_{\Omega^e} (\Omega \otimes_\iota A \otimes_\iota\Omega))
\ar[d]_-{\cong}^-{H^{\bullet}(\id\otimes(\id\otimes\gamma^\vee[-d]\otimes \id))} \ar[r]_-{\cong}^-{\id \otimes \Phi} &
H^\bullet(\Omega \otimes_{\Omega^e} \Hom_{\Omega^e}(\Omega \otimes_\iota A^\vee \otimes_\iota\Omega,\Omega^e))
\ar[dd]_-{\cong}^-{H^\bullet(\id \otimes \tilde{\nu}^*)}\\
H^{\bullet}(\Omega\otimes(\Omega\otimes_{\iota}P^\vee[-d]\otimes_{\iota}\Omega))\ar[d]_-{\cong}^-{(H^\bullet(\id \otimes (\id \otimes \tilde{\mu}^\vee[-d] \otimes \id)))^{-1}} & \\
H^\bullet(\Omega \otimes_{\Omega^e} (\Omega \otimes_\iota A^\vee[-d] \otimes_\iota\Omega)) \ar[d]_-{\cong}^-{H^\bullet(\id \otimes \tilde{\nu})} &
H^\bullet(\Omega \otimes_{\Omega^e} \Hom_{\Omega^e}(\Omega \otimes_\pi B\Omega A^\vee \otimes_\pi\Omega,\Omega^e)) \ar[d]_-{\cong}\\
H^\bullet(\Omega \otimes_{\Omega^e} (\Omega \otimes_\pi B\Omega A^\vee \otimes_\pi\Omega)[-d]) \ar[d]_{\cong} &
H^\bullet(\Hom_{\Omega^e}(\Omega \otimes_\pi B\Omega A^\vee \otimes_\pi\Omega,\Omega)) \ar[d]_{\cong} \\
H^\bullet(\Omega \otimes^{\pi}_{K^e} B\Omega A^\vee[-d]) \ar@{=}[d]  & H^\bullet(\Hom_{K^e}^{\pi}(B\Omega A^\vee,\Omega)) \ar@{=}[d]\\
HH_{d-\bullet}(\Omega) \ar@{.>}[r]_-{\cong}^-{D} \ar[d]^-{-B} & HH^\bullet(\Omega) \ar@{.>}[d]^-{\Delta} \\
HH_{d-\bullet}(\Omega) \ar@{.>}[r]_-{\cong}^-{D} & HH^\bullet(\Omega).
}$$}
Now we consider the following cubic diagram
$$\xymatrixrowsep{4pc}\xymatrixcolsep{5.5pc}\xymatrix@!0{
& HH_{\bullet-d}(A)^\vee \ar[dl]_-{D}^-{\cong} \ar[rr]^(0.4){-B^\vee} \ar@{->}'[d][dd]^(0.4){h_\bullet}
& & HH_{\bullet-d}(A)^\vee \ar[dl]_-{D}^-{\cong} \ar[dd]^(0.7){h_\bullet} \\
HH^{\bullet}(A) \ar[rr]^(0.4){\Delta} \ar[dd]_(0.7){h^\bullet}
 & & HH^\bullet(A) \ar[dd]^(0.7){h^\bullet} \\
 & HH_{d-\bullet}(\Omega A^\vee) \ar[dl]_-{D}^-{\cong} \ar@{->}'[r]^(0.7){-B}[rr] & & HH_{d-\bullet}(\Omega A^\vee)\ar[dl]_{D}^-{\cong} \\
 HH^{\bullet}(\Omega A^\vee) \ar@{->}[rr]^(0.4){\Delta}
 & & HH^{\bullet}(\Omega A^\vee). }$$
We want to show the front square is commutative.
For this, it is enough to prove the other five squares are commutative.
The up square and the down square are commutative by the choices of operators $\Delta$.
By Theorem \ref{Theorem-HHHH_Connes}, the back square is commutative.

Next, we can check directly that the following diagram is commutative:
$$\xymatrixcolsep{8pc}\xymatrix{
\Hom_{K^e}^\pi(B\Omega A^\vee,\Omega) & \Hom_{K^e}^{\pi}(BA,A) \ar[dl]_-{\psi_*^{-1}\circ\phi} \ar[l]_-{\theta}\ar[ddd]^-{(\gamma^\vee[-d])_*} \\
\Hom_{K^e}^{\iota}(A^\vee,\Omega) \ar[u]^-{\widetilde{\tilde{\rho}^*}} & \\
\Omega\otimes_{K^e}^{\iota}A \ar[u]^-{\cong}\ar[d]_-{\id\otimes \gamma^\vee[-d]} & \\
\Omega\otimes_{K^e}^{\iota} P^\vee[-d] \ar[r]^-{\psi^{\prime}}
& \Hom_{K^e}^\pi(BA,P^\vee[-d]) \\
\Omega\otimes_{K^e}^{\iota}A^\vee[-d] \ar[d]_-{\tilde{\tilde{\nu}}[-d]}\ar[u]^-{\id\otimes \tilde{\mu}^\vee[-d]} & \Hom_{K^e}^\pi(BA,A^\vee[-d]) \ar[d]^-{\cong}\ar[u]_-{(\tilde{\mu}^\vee[-d])_*}\\
\Omega\otimes_{K^e}^\pi B\Omega A^\vee [-d]& (A\otimes_{K^e}^\pi BA)^\vee[-d], \ar[ul]_-{(\psi\otimes \id)^{-1}\circ \omega_{A,BA}^{-1}[-d]} \ar[l]^-{\vartheta[-d]}
}$$
where the map $\psi^{\prime}: \Omega\otimes_{K^e}^{\iota}P\rightarrow \Hom_{K^e}^{\pi}(BA,P)$ is given by
$\psi^\prime (f\otimes p)(b)=(-1)^{|p||b|}\psi(f)(b)\cdot p$ for all $f\in \Omega,\ p\in P$ and $b\in BA$.
After taking cohomologies, the left column in the above diagram from the bottom to the top is just the isomorphism
$D : HH_{d-\bullet}(\Omega A^\vee) \rightarrow HH^{\bullet}(\Omega A^\vee)$,
and the right column from the bottom to the top is just the isomorphism $D : HH_{\bullet-d}(A)^\vee \rightarrow HH^{\bullet}(A)$.
The commutativity of the above diagram implies that the left and the right squares in the cubic diagram are commutative.
Thus the whole cubic diagram is commutative.

Since $\Delta : HH^{\bullet}(A) \rightarrow HH_{\bullet}(A)$ satisfies $\Delta(1_{HH^{\bullet}(A)})=0$,
by the commutativity of the front square of the cubic commutative diagram,
we know $\Delta : HH^{\bullet}(\Omega A^\vee) \rightarrow HH_{\bullet}(\Omega A^\vee)$ satisfies $\Delta(1_{HH^{\bullet}(\Omega A^\vee)})=0$.
It follows from Theorem \ref{Theorem-DerSymBV} that $(HH^{\bullet}(\Omega A^\vee),\cup,\Delta)$ is a BV algebra.
Moreover, by the commutativity of the front square of the cubic commutative diagram again,
we obtain that $(HH^{\bullet}(A),\cup,\Delta)$ and $(HH^{\bullet}(\Omega A^\vee),\cup,\Delta)$ are isomorphic as BV algebras.
\end{proof}

For $d$-symmetric dg $K$-rings and their Koszul duals, we have the following result:

\begin{theorem} \label{Theorem-SymCY-BVIso} Let $A$ be a finite dimensional complete typical $d$-symmetric dg $K$-ring. Then

{\rm (1)} $(HH^\bullet(\Omega A^\vee), \cup, \Delta)$ is a BV algebra;

{\rm (2)} $(HH^\bullet(\Omega A^\vee), \cup, \Delta)$ and $(HH^{\bullet}(A), \cup, \Delta)$ are isomorphic as BV algebras.
\end{theorem}

\begin{proof} By the proof of Theorem \ref{Theorem-SymBV}, we have $\Delta(1_{HH^{\bullet}(A)})=0$.
Then the theorem follows from Theorem \ref{Theorem-DerSymCY-BVIso}. \end{proof}

\begin{remark}{\rm
Since BV algebras are Gerstenhaber algebras, Theorem \ref{Theorem-SymCY-BVIso} implies that,
if $A$ is a finite dimensional complete typical $d$-symmetric dg $K$-ring then
$(HH^{\bullet}(A), \cup, [-,-])$ is isomorphic to $(HH^{\bullet}(\Omega A^\vee), \cup, [-,-])$ as Gerstenhaber algebras.
}\end{remark}

\bigskip

\noindent {\footnotesize {\bf ACKNOWLEDGEMENT.} The authors are sponsored by Project 11571341 NSFC.}

\end{document}